\documentclass[12pt,twoside]{amsart}
\usepackage{amsthm,amsfonts,amssymb,amscd,euscript}

\usepackage[matrix,arrow]{xy}

\author{Florin Ambro} 
\address{Institute of Mathematics ``Simion Stoilow'' of the Romanian
Academy\\
P.O. BOX 1-764, RO-014700 Bucharest\\ 
Romania.}
\email{florin.ambro@imar.ro}

\newcommand{\isoto}{{\overset{\sim}{\rightarrow}}}

\newcommand{\Q}{{\mathbb Q}}
\newcommand{\Z}{{\mathbb Z}}

\newcommand{\R}{{\mathbb R}}

\newcommand{\bP}{{\mathbb P}} % projective space
\newcommand{\bA}{{\mathbb A}} % affine space

\newcommand{\cA}{{\mathcal A}}
\newcommand{\cB}{{\mathcal B}}

\newcommand{\cO}{{\mathcal O}}

\newcommand{\bF}{{\mathbf F}}

\newcommand{\fm}{{\mathfrak m}}

\newcommand{\Aut}{\operatorname{Aut}}
\newcommand{\Bs}{\operatorname{Bs}}

\newcommand{\Conv}{\operatorname{Conv}}

\newcommand{\emb}{\operatorname{emb}}

\newcommand{\Hom}{\operatorname{Hom}}
\newcommand{\id}{\operatorname{id}}

\newcommand{\Int}{\operatorname{int}}
\newcommand{\Ker}{\operatorname{Ker}}

\newcommand{\lcm}{\operatorname{lcm}}

\newcommand{\length}{\operatorname{length}}

\newcommand{\mld}{\operatorname{mld}}

\newcommand{\relint}{\operatorname{relint}}

\newcommand{\vol}{\operatorname{vol}}

\theoremstyle{plain}
\newtheorem{thm}{Theorem}[section]

\newtheorem{lem}[thm]{Lemma}
\newtheorem{cor}[thm]{Corollary}

\theoremstyle{definition}
\newtheorem{defn}[thm]{Definition}

\newtheorem{exmp}[thm]{Example}

\newtheorem{rem}[thm]{Remark}

\theoremstyle{remark}

\begin{document}

\bibliographystyle{amsalpha+}
\title{On toric Fano fibrations}

\dedicatory{Dedicated to Professor Vyacheslav V. Shokurov, on the occasion of his seventieth birthday}

%\date{December 13, 2022}
\maketitle

\begin{abstract} A. Borisov classified into finitely many series the set of isomorphism classes of germs of toric $\Q$-factorial singularities, of fixed dimension and with minimal log discrepancy over the special point bounded from below by a fixed real number. We extend this classification to germs of toric Fano fibrations, possibly not $\Q$-factorial.
As an application, we verify in the toric setting a conjecture proposed by V. V. Shokurov on the existence of bounded complements.
\end{abstract} 

%\tableofcontents

%%%%%%%%%%%%%%%%%%%%%%%%%%%%%%%%%%%%%%%%
%%% Document name: OnToricFanoFibrations.tex
%%% Last modified: December 13, 2021
%%%%%%%%%%%%%%%%%%%%%%%%%%%%%%%%%%%%%%%%



%%%%%%%%%%%%%%%%%%%%%%%%%%%%%%%%%%%%%%%%
%%%%%%%%%%%%%%%%%%%%%%%%%%%%%%%%%%%%%%%%

\section{Introduction}

%%%%%%%%%%%%%%%%%%%%%%%%%%%%%%%%%%%%%%%%
%%%%%%%%%%%%%%%%%%%%%%%%%%%%%%%%%%%%%%%%

 The explicit classification of terminal $3$-fold singularities played an important role in the birational classification of projective $3$-folds. Based on earlier computer-generated attempts to explicitly classify toric terminal $4$-fold singularities, Borisov~\cite{Bor99} obtained a qualitative classification of toric $\Q$-factorial singularities $P\in X$ with $\dim X=d$ and $\mld X\ge t$ (i.e. log discrepancies of geometric valuations of $X$ are at least $t$). The latter correspond to finite subgroups of the real torus $\R^d/\Z^d$ which intersect the set $\{x\in [0,1)^d; \sum_{i=1}^dx_i< t\}$ only at the origin, and Borisov showed that they fall into {\em finitely many series}. Our motivation is to understand these series, find their geometric meaning, and possibly extend them to non-toric singularities. If the latter is indeed true, it could have important applications in the birational classification of projective varieties, since in the toric case the standard (local) conjectures of Shokurov on minimal log discrepancies do follow from the finiteness of series.
 
  Borisov's definition of series simplifies if the global condition $\mld X\ge t$ is replaced by the local condition $\mld_P X\ge t$ (i.e. log discrepancies of geometric valuations of $X$ with center $P$ are at least $t$).
 By Lawrence~\cite{Law91}, the subgroups of $\R^d/\Z^d$ which avoid 
 the image of the open set $\{x\in (0,1)^d; \sum_{i=1}^dx_i< t\}$ have only finitely many maximal elements $G$ with respect to inclusion (called {\em series}), and Borisov showed that a toric $\Q$-factorial affine variety $X$ with (unique) invariant point $P$ satisfies $\dim X=d$ and $\mld_P X\ge t$ if and only if it corresponds to a finite subgroup of $\R^d/\Z^d$ contained in some $G$. We only consider such local series in this paper.

 A first observation is that series of toric $\Q$-factorial singularities have a dual interpretation, which make sense even in the non-$\Q$-factorial case. We identify a series with a closed subgroup $\Z^d\le G\le \R^d$.  A basis of the dual abelian group $G^*\subset (\Z^d)^*$ defines a projection $\Phi\colon \R^d\to \R^p$ such that $G=\Phi^{-1}(\Z^p)$ (i.e. $G=G^{**}$).
 Denote $U=\{x\in \R_{\ge 0}^d;\sum_{i=1}^dx_i\le 1\}$. The condition 
 $G\cap \Int(tU)=\emptyset$ is equivalent to
 $\Z^p\cap \Int(t\Phi(U))=\emptyset$. The series $G$ belongs to a finite family if and only if $t\Phi(U)$ is contained in a bounded box. 
 Therefore Borisov's existence of only finitely many local series can be restated in terms of projections as follows:
 an affine toric $\Q$-factorial variety $P\in X=T_N\emb(\sigma)$ with $\dim X=d$ satisfies $\mld_P X\ge t$ if and only if there exists a surjective homomorphism of lattices $\Phi\colon N\to N'$ such that 
 $N'\cap \Int(t\Phi(U))=\emptyset$ and the pair $(N',t\Phi(U))$ is bounded
 (i.e. the convex set is contained in a bounded box after choosing some basis of the lattice). Here $U=\Conv(\{0\}\cup\sigma(1))$ coincides with the polar set of the anti-canonical moment polytope $\square_{-K_X}\subset M_\R$, and is characterized in terms of valuations as follows: the primitive lattice points in the interior of $hU$ are in one to one correspondence with toric valuations $E$ of $X$ with center $P$ such that the log discrepancy $a_E(X)$ is strictly less than $h$. This statement makes sense and holds even if $X$ is not $\Q$-factorial.
 
 We call $\Phi\colon (N,tU)\to (N',t\Phi(U))$ a {\em $t$-lc reduction} of $P\in X$. Whereas $0$ is an extremal point of $U$, $0$ is only a boundary point of $\Phi(U)$. Therefore $(N',t\Phi(U))$ may not be associated to a toric germ, but rather to a germ of toric proper fibration $f'\colon (X',B')\to Y'\ni P'$
 where $\dim X'=\dim N'$, $B'$ is an invariant boundary with standard coefficients, $Y'$ affine with invariant point $P'$, $-K_{X'}-B'$ is $f'$-semiample 
 and $\mld_{{f'}^{-1}P'}(X',B')\ge t$. Moreover, $\Phi(U)=U'$ admits three equivalent definitions, same as $U$. We will show that $t$-lc reductions exist in this setting too. Our main result is 
 
 \begin{thm}\label{mth}
 Fix $d\ge 1$, $t>0$ and a DCC set $\cB\subset [0,1]$. Let $f\colon X\to Y\ni P$ be a toric fibration with affine base $Y$ having positive dimension and invariant point $P$. Suppose $\dim X=d$. Let $\sum_i E_i$ be the complement of the torus inside $X$, let $B=\sum_i b_i E_i$ with $b_i\in \cB$ for all $i$. Write $X=T_N\emb(\Delta)$, let $\square=\square_{-K-B}\subset M_\R$ be the moment polytope of the torus invariant $\R$-Weil divisor $-K-B$ and $U=\square^* \subset N_\R$ its polar set. The following are equivalent:
 	\begin{itemize}
 		\item[a)] There exists an open neighborhood $P\in V\subseteq Y$ and a boundary $B^+_{f^{-1}V}\ge B|_{f^{-1}V}$ on $f^{-1}V$ such that  
 		$K_{f^{-1}V}+B^+_{f^{-1}V}\sim_\R 0$ and $\mld_{f^{-1}P}(f^{-1}V,B^+_{f^{-1}V})\ge t$.
 		\item[b)] There exists a non-zero surjective homomorphism of lattices $\Phi\colon N\to N'$ such that the image $U'=\Phi(U)\subset N'_\R$
 		is a compact polytope, $N'\cap \Int(tU')=\emptyset$ and $(N',tU')$ is bounded.
 	\end{itemize}
 \end{thm}

A stronger statement holds in fact, where we only require that $\dim X$ is fixed and the ratios $\frac{1-b_i}{t}$ belong to a fixed ACC set (Theorem~\ref{bndser}).

Since $Y$ is affine toric, a) holds if and only if it holds with $V=Y$. If some coefficient of $B$ is $1$ then $U$ is not compact, hence the compact 
condition in b) is not redundant. We say that $B$ has {\em $r$-hyperstandard coefficients} if they are of the form $1-\frac{x}{q}$, where $x\in \frac{1}{r}\Z\cap [0,1]$ and $q\in \Z_{\ge 1}$. If $B$ has $r$-hyperstandard coefficients and $t$ is fixed and rational, Theorem~\ref{mth} can be deduced from Nill-Ziegler's result~\cite{NZ11} stating that polytopes in $\R^d$ with extremal points rational of bounded index and without interior lattice points are either 
finite (up to lattice translations and change of basis of the lattice), or they project onto a smaller dimensional polytope with the same properties.
In general, Theorem~\ref{mth} follows from a generalization of Nill-Ziegler's theorem in our setting (Theorem~\ref{SerPoly}).

One application of Theorem~\ref{mth} is the toric case of a conjecture of Shokurov:

\begin{thm}\label{mapp}
Let $f\colon X\to Y$ be a toric fibration with affine base $Y$, let $B$ be an invariant boundary on $X$ with $r$-hyperstandard coefficients, suppose 
$-K-B$ is $f$-semiample. Let $t>0$.
\begin{itemize}
	\item[a)] Suppose $Y$ has an invariant closed point $P$ and $\mld_{f^{-1}P}(X,B)\ge t$. Then there exists $B^+\ge B$ such that $n(K+B^+)\sim 0$ and 
	$\mld_{f^{-1}P}(X,B^+)\ge t$, where $n$ is a positive integer depending only on $\dim X,r,t$.
	\item[b)] Suppose $\mld(X,B)\ge t$ and $t<1$. Then there exists $B^+\ge B$ such that $n(K+B^+)\sim 0$ and 
	$\mld(X,B^+)\ge t$, where $n$ is a positive integer depending only on $\dim X,r,t$.
\end{itemize} 
\end{thm}

If $f$ is an isomorphism and $X$ is $\Q$-factorial, Theorem~\ref{mapp}.b) was proved by Shokurov~\cite{Sho04}, based on Borisov's result. A slightly more general (but equivalent modulo toric MMP) statement is proved in Section 4.
We do not address here the more general conjecture on the existence of bounded $a-n-$complements for boundaries with arbitrary coefficients (see ~\cite[Conjecture 5 on page 212]{Sho20}).

We outline the structure of this paper. In Section 2 we prove the main combinatorial result (Theorem~\ref{SerPoly}), a generalization of Nill-Ziegler's theorem. The proof is similar but more direct, based on the reduced projections of Kannan-Lovasz and a generalization of a result of 
Pikhurko (Theorem~\ref{GenPik}). In Section 3 we translate the combinatorial result in the setting of germs of toric fibrations which admit complements with good singularities. We also investigate basic properties of the image of a $t$-lc reduction, including the existence of hyperplane sections on the base with good singularities (Theorem~\ref{bndmult}). In Section 4 we assume the boundary has $r$-hyperstandard coefficients, and construct a complement with bounded index and good singularities. In Section 5 we specialize our main result to the case of germs of toric log singularities
$P\in (X,B)$. We also show that if $\dim X$ is fixed, if the coefficients of $B$ and $\mld_P(X,B)$ belong to a DCC set, then they are finite.

%\clearpage
%%%%%%%%%%%%%%%%%%%%%%%%%%%%%%%%%%%%%%%%
%%%%%%%%%%%%%%%%%%%%%%%%%%%%%%%%%%%%%%%%

\section{Geometry of numbers}

%%%%%%%%%%%%%%%%%%%%%%%%%%%%%%%%%%%%%%%%
%%%%%%%%%%%%%%%%%%%%%%%%%%%%%%%%%%%%%%%%

%%%%%%%%%%%%%%%%%%%%%%%%%%%%%%%%%%%%%%%%

\subsection{Preliminary}

%%%%%%%%%%%%%%%%%%%%%%%%%%%%%%%%%%%%%%%%

Let $V\simeq \R^d$ be a finite dimensional $\R$-vector space. Let
$V^*=\Hom_\R(V,\R)$ be the dual $\R$-vector space, with induced duality pairing $\langle \cdot,\cdot\rangle\colon V^*\times V\to \R$.
Let $\square\subset V$ be a closed convex set which contains the origin. Its {\em polar set} is defined by 
$$
\square^*=\{v^*\in V^*;\langle v^*,v\rangle\ge -1 \ \forall v\in \square\}.
$$
It is again a closed convex set in $V^*$, which contains the origin. We have $\square=(\square^*)^*$. Indeed, this follows from the duality theorem for closed convex cones applied to the cone in $V\times \R$ generated by the set $\square\times 1$.

If $\square\subset V$ is a closed convex set and $\varphi\in V^*$, we denote $h_\square(\varphi)=\inf\langle \varphi,\square\rangle\in \R\cup\{-\infty\}$, where 
$\langle \varphi,\square\rangle$ is the image interval $\{\varphi(v);v\in \square\}$. If $\square$ is polyhedral, the infimum is either $-\infty$ or a minimum.

If $\square\subset V$ is a closed convex set, the difference $\square-\square=\{x-y;x,y\in \square\}\subset V$ is a closed convex set
symmetric about the origin, and $(\square-\square)^*$ consists of the linear forms $\varphi\in V^*$ such that the interval $\langle \varphi,\square\rangle\subset \R$ has length at most $1$.

Suppose $\square\subset V$ is a compact convex set of dimension $d$. For 
$P\in \square$, denote by $\gamma(P\in \square)$ the maximal $t\ge 0$
such that $P+t(\square-\square)\subseteq \square$. We have $0\le \gamma(P\in \square)\le \frac{1}{2}$, and the left hand side (resp. right hand side) inequality is attained only if $P\in \partial \square$ (resp. $\square$ is symmetric about $P$). We also have $\gamma(P\in \square)=\frac{\epsilon}{1+\epsilon}$, where $\epsilon$ is maximal with 
the property $-\epsilon(\square-P)\subseteq \square-P$.

If $\square$ is a simplex with vertices $v_0,\ldots,v_d$, denoted
$\square=[v_0,\ldots,v_d]$, any point $P\in \square$ admits a unique 
representation $P=\sum_{i=0}^dt_iv_i$ with $\min_it_i\ge 0$ and 
$\sum_{i=0}^dt_i=1$. Then $\gamma(P\in \square)=\min_{i=0}^dt_i$.
Moreover, we will denote $t_i$ by $\gamma_{v_i}(P\in [v_0,\ldots,v_d])$.

If $\square$ is a polytope, we will denote by $\square(0)$ the set of extremal points of $\square$, i.e. the finite set of its vertices.

\begin{lem}
	Let $S=[v_0,\ldots, v_d]$ be a $d$-dimensional simplex in $\R^d$.
	Let $P\in \Int S$ and suppose $P\ne 0$. Let $\lambda>1$ be maximal such that $\lambda P\in S$. Then $\lambda P$ belongs to the relative interior of a proper face $[v_i;i\in I]$ of $S$, and $(v_i)_{i\in I}$ are linearly independent.
\end{lem}

\begin{proof}
	Write $P=\sum_{i=0}^dt_iv_i$ with $t_i>0$ for all $i$ and $\sum_{i=0}^dt_i=1$. Since $v_1-v_0,\ldots,v_d-v_0$ form a basis of $\R^d$, there exists a unique solution for the system 
	$\sum_{i=0}^d\alpha_i v_i=0$, $\sum_{i=0}^d\alpha_i=1$.
	Moreover, any non-trivial relation among $v_i$ is a multiple of 
	$(\alpha_0,\ldots,\alpha_d)\in \R^{d+1}$.
	We have 
	$$
	tP=\sum_{i=0}^d(tt_i-(t-1)\alpha_i)v_i, \ \sum_{i=0}^d (tt_i-(t-1)\alpha_i)=1.
	$$
	Therefore $tP\in S$ if and only if $tt_i-(t-1)\alpha_i\ge 0$
	for all $i$. If $t_i\ge \alpha_i$ for all $i$, then $P=0$. Since we assumed otherwise, $t_i<\alpha_i$ for some $i$. Therefore the maximal $\lambda$ is well defined, and $\lambda-1=\min_{\alpha_i>t_i}\frac{t_i}{\alpha_i-t_i}$. We have $\lambda>1$, $\lambda P=\sum_{i=0}^d \lambda_iv_i$ and $\lambda_i=\lambda t_i-(\lambda-1)\alpha_i$.
	Say $\lambda_0,\ldots,\lambda_k>0=\lambda_{> k}$. Here $0\le k<d$.
	We have $\lambda P\in \relint [v_0,\ldots,v_k]$.
	Suppose by contradiction that $v_0,\ldots,v_k$ are linearly dependent.
	Then $\alpha_{>k}=0$. But $0=\lambda_d=\lambda t_d-(\lambda-1)\alpha_d$,
	hence $\alpha_d>0$. Contradiction!
\end{proof}

\begin{lem}\label{cecr}
	Let $\pi\colon \R^d \to \R^{d'}$ be a surjective linear homomorphism. Let $\square\subset \R^d$ be a closed convex set of 
	dimension $d$, containing $0$ in its interior. Then $\square':=\pi(\square)$ is a closed convex set in $\R^{d'}$ containing the origin in its interior. Let $V_0=\Ker \pi$ and $\square_0:=\square\cap V_0$. Then $\square_0\subset V_0$ is a closed convex set of dimension $d-d'$, containing $0$ in its interior. Denote $\gamma=\gamma(0\in \square)$, $\gamma'=\gamma(0\in \square')$ and $\gamma_0=\gamma(0\in \square_0)$. Then 
	$$
	\gamma_0\gamma'\le \gamma\le \min(\gamma_0,\gamma').
	$$
\end{lem}

\begin{proof}
	The second inequality is easy. For the first inequality, choose $0\ne v\in \square$. Then $\pi(v)\in \square'$. 
	
	Suppose $\pi(v)=0$. Then $\R v\cap \square=\R v\cap \square_0$, hence $-\epsilon v\in \square$ for some 
	$\epsilon>0$ with $\frac{\epsilon}{1+\epsilon}\ge \gamma_0$.
	
	Suppose $\pi(v)\ne 0$. Let $\epsilon'>0$ be maximal such that $-\epsilon'\pi(v)\in \square'$. We have $\frac{\epsilon'}{1+\epsilon'}\ge \gamma'$. There exists $w\in \square$ such that $\pi(w)=-\epsilon'\pi(v)$. This means that 
	$$
	v_0:=\frac{1}{1+\epsilon'}w+\frac{\epsilon'}{1+\epsilon'}v\in \square_0
	$$
	If $v_0=0$, then $-\epsilon'v=w\in \square$. Suppose $v_0\ne 0$.
	Let $\epsilon_0>0$ be maximal such that $-\epsilon_0v_0\in \square_0$.
	We have $\frac{\epsilon_0}{1+\epsilon_0}\ge \gamma_0$. If we denote 
	$
	\epsilon:=\frac{\epsilon'\epsilon_0}{1+\epsilon'+\epsilon_0},
	$
	the following identity holds:
	$$
	-\epsilon v=\frac{\epsilon}{\epsilon'}w+(1-\frac{\epsilon}{\epsilon'})(-\epsilon_0v_0).
	$$
	Since $w,-\epsilon_0v_0\in \square$, we deduce $-\epsilon v\in \square$. We compute
	$$
	\frac{\epsilon}{1+\epsilon}=\frac{\epsilon'}{1+\epsilon'}\cdot \frac{\epsilon_0}{1+\epsilon_0}\ge \gamma'\cdot \gamma_0.
	$$
\end{proof}

\begin{rem}
	Both inequalities are attained. For the left hand side, let $T=[A,B,C]\subset \R^2$ be a triangle containing $0$ in its interior. Let $\pi\colon \R^2\to \R$ be the projection which maps $B$ and $C$ onto the same point $P$. 
	Let $P'\in \R$ be the image of $A$. Let $0=(1-\lambda)P'+\lambda P$.
	The intersection of $T$ with $\Ker \pi$ is an interval $[B',C']$.
	Let $0=(1-\mu)B'+\mu C'$. Then $0\in T$ has the following barycentric coordinates:
	$$
	0=(1-\lambda)A+\lambda(1-\mu)B+\lambda\mu C.
	$$
\end{rem}

Let $\Lambda\simeq \Z^d$ be a $d$-dimensional lattice. Let $\square\subset \Lambda_\R$ be a compact convex set of dimension $d$. For $1\le i\le d$,
denote by $\lambda_i(\Lambda,\square-\square)$ the $i$-th successive minimum of Minkowski, i.e. the smallest real number $t>0$ such that $\Lambda \cap t(\square-\square)$ contains $i$ linearly independent elements.

If $\Lambda\cap \Int\square\ne \emptyset$, we define the {\em Pikhurko constant} of $(\Lambda,\square)$ as follows:
$$
\gamma(\Lambda,\square):=\max_{P\in \Lambda\cap \Int \square}\gamma(P\in \square).
$$

\begin{defn}
	We say that $(\Lambda,\square)$ is {\em bounded} if one of the following equivalent conditions holds:
	\begin{itemize}
		\item[a)] $\lambda_1(\Lambda,\square-\square)$ is bounded away from zero.
		\item[b)] $\lambda_d(\Lambda^*,(\square-\square)^*)$ is bounded above.
		\item[c)] There exists an isomorphism $\Phi\colon \Lambda\isoto \Z^d$
		such that $\Phi_\R(\square)\subseteq \prod_{i=1}^d[x_i,x_i+l]$, where 
		$x_i\in \R$, and $l$ is a positive number bounded above.
	\end{itemize}
\end{defn}

Properties a) and b) are equivalent due to the inequality
$$
1\le \lambda_1(\Lambda,\square-\square)\cdot \lambda_d(\Lambda^*,(\square-\square)^*) \le d!,
$$ 
which is one of the inequalities of Mahler's Transference Theorem~\cite[page 219, thm VI]{Cas97}. If c) holds, let $\varphi_1,\ldots,\varphi_d$ be the components of $\Phi$. Then 
they form a basis of $\Lambda^*$ and $\length \varphi_i(\square)\le l$ for all $i$. Therefore $\lambda_d(\Lambda^*,(\square-\square)^*)\le l$.
Suppose now b) holds, i.e. $\lambda_d(\Lambda^*,(\square-\square)^*)\le c$.
This means that there exist $\varphi_1,\ldots,\varphi_d\in \Lambda^*$, 
linearly independent, such that $\length \varphi_i(\square)\le c$ for all $i$. By an argument of Mahler~\cite[page 68]{GL87}, $\Lambda^*$ admits a basis $\varphi'_1,\ldots,\varphi'_d$ such that $\length \varphi'_i(\square)\le dc$ for all $i$. Then $(\varphi'_1,\ldots,\varphi'_d)\colon \Lambda\isoto \Z^d$ is an isomorphism
as in c) and the image of $\square$ is contained in a product of intervals of length at most $dc$.

Note that in case c) we have $\sqrt[d]{\vol_\Lambda\square}\le l$. By the above argument, $\sqrt[d]{\vol_\Lambda\square}\le d\cdot \lambda_d(\Lambda^*,(\square-\square)^*)$.

Note that if $\square$ contains $0$, then c) implies 
$\Phi_\R(\square)\subseteq [-l,l]^d$. If moreover $\square$ is a polytope whose vertices are rational points of bounded index, then $\Phi_\R(\square)$ belongs to a finite set of polytopes.

%%%%%%%%%%%%%%%%%%%%%%%%%%%%%%%%%%%%%%%

\subsection{Central integer points of intervals}

%%%%%%%%%%%%%%%%%%%%%%%%%%%%%%%%%%%%%%%

Let $I\subset \R$ be a compact interval such that $\Z\cap \Int I\ne \emptyset$. We compute and estimate the Pikhurko constant $\gamma:=\gamma(\Z,I)$.

a) Case of odd number of integer interior points: $\Z\cap \Int I=\{z,z+1,\ldots,z+2k\}$ for some integer $k\ge 0$. Then $I=[z-\alpha,z+2k+\beta]$ for some 
$\alpha,\beta\in (0,1]$. After possibly replacing $I$ by $-I$, we may suppose $\alpha\ge \beta$. Then 
$\gamma$ is attained by $P=z+k$, equal to 
$$
\gamma=\frac{k+\beta}{2k+\alpha+\beta}.
$$
Keeping the set of interior integer points fixed, this invariant is minimal when $\alpha=1$, hence 
$$
\gamma\ge \frac{k+\beta}{2k+1+\beta}.
$$

b) Case of even number of integer interior points: $\Z\cap \Int I=\{z,z+1,\ldots,z+2k+1\}$ for some integer $k\ge 0$. Then $I=[z-\alpha,z+2k+1+\beta]$ for some 
$\alpha,\beta\in (0,1]$. After possibly replacing $I$ by $-I$, we may suppose $\alpha\ge \beta$. Then 
$\gamma$ is attained by $P=z+k$, equal to 
$$
\gamma=\frac{k+\alpha}{2k+1+\alpha+\beta}.
$$
Keeping the set of interior integer points fixed, this invariant is minimal when $\alpha=\beta$, hence 
$$
\gamma\ge \frac{k+\beta}{2k+1+2\beta}.
$$

We deduce that $\gamma>\frac{k}{2k+1}$. Therefore $\gamma$ approaches $\frac{1}{2}$ from below as the length of $I$ becomes arbitrarily large.

\begin{lem}\label{lig} 
	$(1-2\gamma)\cdot \length I\le 1$. 
\end{lem}

\begin{proof}
	In case a), $(1-2\gamma)(2k+\alpha+\beta)=\alpha-\beta<1$.
	In case b), $(1-2\gamma)(2k+1+\alpha+\beta)=
	1-\alpha+\beta \le 1$.
\end{proof}

The inequality in Lemma~\ref{lig} can be restated as 
$\lambda_1(\Z,I-I)+2\gamma(\Z,I)\ge 1$, since $\lambda_1(\Z,I-I)=(\length I)^{-1}$ in this case.

\begin{lem}[Pikhurko] Suppose $I(0)\subset \frac{1}{l}\Z$. Then $\gamma\ge \frac{1}{l+2}$, attained up to lattice translation only by the interval $I=[-\frac{1}{l},1+\frac{1}{l}]$ and its two interior integer points $\{0,1\}$.
\end{lem}

\begin{proof}
	Case a): May suppose $\alpha=1$. Then $\gamma\ge \frac{k+\beta}{2k+1+\beta}\ge \frac{kl+1}{(2k+1)l+1}$.
	The smallest value is taken for $k=0$, equal to $\frac{1}{l+1}$.
	Case b): May suppose $\alpha=\beta$. Then $\gamma\ge \frac{k+\beta}{2k+1+2\beta}\ge \frac{kl+1}{(2k+1)l+2}$.
	The smallest value is taken for $k=0$, equal to $\frac{1}{l+2}$.
\end{proof}

\begin{lem}\label{gP1}
	Let $\cA\subset (0,+\infty)$ be an ACC set. Suppose $0\in I$ and 
	$I(0)\subset \{0\}\cup\{\pm \frac{1}{a};a\in \cA \}$. Then $\gamma$ is bounded away from zero.
\end{lem}

\begin{proof}
	If $\Z\cap \Int I$ contains at least $3$ points, then $\gamma>\frac{1}{3}$.
	So we may suppose $\Z\cap \Int I$ has at most two elements.
	
	1) Case $0$ is an extremal point of $I$. We may suppose $I=[0,\frac{1}{a}]$ for some $a\in \cA$.
	
	Case a): $\Z\cap \Int I=\{1,\ldots,2k+1\}$ and $2k+1<\frac{1}{a}\le 2k+2$.
	Here $\gamma$ is attained by $P=k+1$, equal to $\gamma=1-a(k+1)$.
	If $k=0$, we obtain $\gamma=1-a$. If $k\ge 1$, $\gamma>\frac{1}{3}$.
	
	Case b): $\Z\cap \Int I=\{1,\ldots,2k+2\}$ and $2k+2<\frac{1}{a}\le 2k+3$.
	Here $\gamma$ is attained by $P=k+1$, equal to $\gamma=a(k+1)$.
	We have $\gamma\ge \frac{k+1}{2k+3}\ge \frac{1}{3}$.
	
	We conclude in this case 
	$$
	\gamma\ge \min(1-a,\frac{1}{3}).
	$$
	
	2) Case $0$ is an interior point of $I$. 
	
	Case a): $\Z\cap \Int I=\{z\}$. Then $z=0$ and $\gamma\ge \frac{\beta}{1+\beta}$. Have $\beta=\frac{1}{a}$ for some $a\in \cA$. Then 
	$\gamma\ge \frac{1}{1+a}$. 
	
	Case b): $\Z\cap \Int I=\{z,z+1\}$. Here may suppose $\alpha=\beta$ and $\gamma=\frac{\beta}{1+2\beta}$. There are two possibilities:
	
	Suppose $z=0$. Then $1+\beta=\frac{1}{a}$ for some $a\in \cA\cap [\frac{1}{2},1)$. Then $\gamma=\frac{1-a}{2-a}$.
	
	Suppose $z+1=0$. Then $\beta=\frac{1}{a}$ for some $a\in \cA\cap [1,+\infty)$. Then $\gamma=\frac{1}{2+a}$.
	
	We conclude in case 2) that
	$$
	\gamma\ge \min(\frac{1-a}{2-a},\frac{1}{2+a}).
	$$
	Alltogether, we obtain 
	$$
	\gamma\ge \min(\frac{1}{3}, \frac{1-a}{2-a},\frac{1}{2+a}).
	$$
\end{proof}

%\clearpage
%%%%%%%%%%%%%%%%%%%%%%%%%%%%%%%%%%%%%%%

\subsection{A generalization of Pikhurko's Theorem~\cite{Pikh01}}

%%%%%%%%%%%%%%%%%%%%%%%%%%%%%%%%%%%%%%%

\begin{thm}\label{GenPik} Fix $d\in \Z_{\ge 1}$ and an ACC set $\cA\subset (0,+\infty)$. 
	Let $\Lambda\simeq \Z^d$ be a lattice of dimension $d$, let $\square\subset \Lambda_\R$ be a compact polytope of dimension $d$. Suppose $0\in \square$ and every non-zero extremal point of $\square$ is of the form $v=\frac{e}{a_e}$ where $e\in \Lambda^{prim}$ and $a_e\in \cA$.
	If $\Lambda\cap \Int\square\ne \emptyset$,
	then $\gamma(\Lambda,\square)$ is bounded away from zero.
\end{thm}

\begin{lem}\label{BndPik} 
	In the setting of Theorem~\ref{GenPik}, suppose moreover that $(\Lambda,\square)$ is bounded.
	Then $\min_{P\in \Lambda\cap \Int\square}\gamma(P\in \square)$ is bounded away from zero.
\end{lem}

\begin{proof} Fix $P\in \Lambda\cap \Int\square$ and denote $\gamma=\gamma(P\in \square)$. There exists $v_0\in \square(0)$ such that $P=\gamma v_0+(1-\gamma)v'_0$
	and $v'_0\in \partial \square$. Thus $v'_0$ belongs to a face $F$ of $\square$ of dimension at most $d-1$. By Caratheodory's theorem, there exist $v_1,\ldots,v_p\in F(0)\subset \square(0)$, for some $1\le p\le d$, forming the vertices of a $(p-1)$-dimensional simplex which contains $v'_0$ in its relative interior. Then $v_0,v_1,\ldots,v_p$ are the vertices of a $p$-dimensional simplex containing $P$ in its relative interior, and
	$\gamma=\gamma_{v_0}(P\in [v_0,\ldots,v_p])$. Thus we may write 
	$P=\sum_{i=0}^p t_iv_i$ with $\min_it_i>0$, $\sum_i t_i=1$ and $t_0=\gamma$.
	
	Suppose $v_0,v_1,\ldots,v_p$ are linearly independent. That is,
	$0,v_0,\ldots,v_p$ are the vertices of a simplex of dimension $p+1$,
	$P$ lies in the relative interior of the face opposite the origin, with
	$P=\sum_{i=0}^p t_iv_i$ with $\min_it_i>0$, $\sum_i t_i=1$ and $t_0=\gamma$. {\bf (Case 0)}. 
	
	Suppose from now $v_0,v_1,\ldots,v_p$ are linearly dependent.
	As they are the vertices of a $p$-dimensional simplex, there exists a unique relation $0=\sum_{i=0}^p\alpha_iv_i$ with $1=\sum_{i=0}^p\alpha_i$. Moreover, any non-trivial relation among the $v_0,v_1,\ldots,v_p$ is a non-zero multiple of $(\alpha_0,\ldots,\alpha_p)$.
	
	We claim that $\alpha_0\ge 0$. Indeed, there exists $e\in \Lambda^*_\R\setminus 0$ such that $\min\langle \square,e\rangle=\langle v'_0,e\rangle=:c$. Since $0\in \square$, we have $c\le 0$. Since $v'_0$ is contained in the relative interior of the simplex $[v_1,\ldots,v_p]\subset \square$, we obtain $\langle v_i,e\rangle=c$ for all $1\le i\le p$. Since $P\in \Int \square$, we have $\langle P,e\rangle>c$. That is, $\langle v_0,e\rangle>c$. We obtain
	$$
	\alpha_0=\frac{-c}{\langle v_0,e\rangle-c}\ge 0.
	$$
	
	\underline{Suppose $\alpha_0=0$}. Then $0=\sum_{i=1}^p\alpha_iv_i$,
	$1=\sum_{i=1}^p\alpha_i$, and $v'_0$ lies in the relative interior of the simplex with vertices $v_1,\ldots,v_p$.
	
	If $v'_0=0$, then $P=\gamma v_0+(1-\gamma)\cdot 0$ {\bf (Case 1a)}.
	
	Suppose $v'_0\ne 0$. Then there exists $\lambda>1$ maximal with the property that $\lambda v'_0\in [v_1,\ldots,v_p]$. After reordering,
	$\lambda v'_0$ belongs to a simplex face $[v_1,\ldots,v_q]$, and 
	$v_1,\ldots,v_q$ are linearly independent. Then $v_0,v_1,\ldots,v_q,0$ are the vertices of a simplex containing $P$ in its interior, and $\gamma=\gamma_{v_0}(P\in [v_0,v_1,\ldots,v_q,0])$.
	{\bf (Case 1b)}
	
	\underline{Suppose $\alpha_0>0$}.
	
	Suppose $P=0$. Then $P=0$ lies in the relative interior of the simplex with vertices $v_0,\ldots,v_p$, with barycentric coordinates $\alpha_0,\ldots,\alpha_p$. Thus $\gamma=\gamma_{v_0}(P\in [v_0,\ldots,v_p])$. {\bf (Case 2)} 
	
	Suppose from now that $P\ne 0$. Let $\lambda>1$ be maximal such that $\lambda P\in [v_0,v_1,\ldots,v_p]$. Then
	$$
	\lambda P=\sum_{i=0}^p\lambda_iv_i, \lambda_i=\lambda t_i-(\lambda-1)\alpha_i.
	$$ 
	Let $I=\{i\in \{0,\ldots,p\}; \lambda_i>0\}$. Then $I\subsetneq \{0,1,\ldots,p\}$ and $(\lambda_i)_{i\in I}$ are linearly independent.
	We obtain $\lambda P=\sum_{i\in I}\lambda_i v_i$ with $\lambda_i>0$ for all $i\in I$ and $\sum_{i\in I}\lambda_i=1$. 
	
	Subcase $0\in I$. Say $I=\{0,1,\ldots,k\}$, for some $0\le k<p$.
	Then $\lambda P=\sum_{i=0}^k\lambda_i v_i$ with all $\lambda_i>0$ and 
	$\sum_{i=0}^kv_i=1$. Then $P=(1-\frac{1}{\lambda})\cdot 0+\sum_{i=0}^k\frac{\lambda_i}{\lambda}v_i$.
	So $P$ lies in the relative interior of the simplex with vertices
	$0,v_1,\ldots,v_k$. {\bf (Case 3)}.
	From $\gamma=t_0$ and $\lambda_0=\lambda t_0-(\lambda-1)\alpha_0$, we obtain 
	$$
	\gamma =t_0=\frac{\lambda_0+(\lambda-1)\alpha_0}{\lambda}>\frac{\lambda_0}{\lambda}=\gamma_{v_0}(P\in [0,v_1,\ldots,v_k]).
	$$
	
	Subcase $0\notin I$. Then $\lambda P\in \relint [v_1,\ldots,v_k]$
	for some $1\le k\le p$. {\bf (Case 4)}
	Here $0,v_1,\ldots,v_k$ are the vertices of a $p$-dimensional simplex containing $P$ in its interior, with barycentric coordinates
	$$
	P=(1-\frac{1}{\lambda})\cdot 0+\sum_{i=1}^k\frac{\lambda_i}{\lambda}v_i.
	$$
	Since $0\notin I$, we have $\lambda_0=0$. That is $0=\lambda t_0-(\lambda-1)\alpha_0$. Therefore $t_0=(1-\frac{1}{\lambda})\alpha_0$, i.e.
	$$
	\gamma=\alpha_0\cdot \gamma_0(P\in [0,v_1,\ldots,v_p]).
	$$
	
	Now we use the boundedness and ACC assumption. Let $v=\frac{e}{a_e}$ be a non-zero vertex of $\square$. By assumption, $v$ is contained in a bounded box. Since $a_e$ satisfies ACC, it is bounded above. It follows that $e$ belongs to a bounded box, hence finite. Therefore $a_e$ is bounded away from zero. Moreover, the lattice points of $\square$ belong to a given finite set.
	
	We show that $\gamma$ is away from zero, in each of the five cases:
	
	Case 0): Here $v_i\ne 0$ for all $0\le i\le p$, and $P=\sum_{i=0}^px_ie_i$ with $x_i\in \Q_{>0}$ finite.
	Then $t_i=\frac{a_ix_i}{a_0x_0+\cdots+a_px_p}$ for $0\le i\le p$.
	Therefore 
	$$
	\gamma=\frac{a_0x_0}{a_0x_0+\cdots+a_px_p}.
	$$
	Since $a_i$ satisfy ACC, they are bounded above. Then
	$\sum_{i=0}^pa_ix_i\le (\max_i a_i)\sum_{i=0}^px_i$ is bounded above.
	Since $a_0$ is bounded away from zero and $x_0$ is finite, it follows that $\gamma$ is bounded away from zero.
	
	Case 1): Here $P$ is contained in the relative interior of a simplex with vertices $0,v_1,\ldots,v_r$, and $\gamma$ is one of the barycentric coordinates. Write $v_i=\frac{e_i}{a_i}$. Then $P=\sum_{i=1}^rx_ie_i$,
	where $x_i\in \Q_{>0}$ are finite and $\sum_{i=1}^ra_ix_i<1$. We have 
	$$
	P=(1-\sum_{i=1}^ra_ix_i)\cdot 0+\sum_{i=1}^ra_ix_iv_i.
	$$
	If $\gamma=a_ix_i$, it is away from zero from above. If $\gamma=1-\sum_{i=1}^ra_ix_i$, it satisfies DCC since $a_i$ satisfy ACC and $x_i$ are finite. Therefore $\gamma$ is away from zero.
	
	Case 2): Here $v_i\ne 0$ for all $0\le i\le p$, and $P=0=\sum_{i=0}^px_ie_i$ with $x_i\in \Q_{>0}$ finite. We have 
	$$
	P=\sum_{i=0}^p\frac{a_ix_i}{a_0x_0+\cdots+a_px_p}v_i.
	$$
	So $\gamma=\frac{a_0x_0}{a_0x_0+\cdots+a_px_p}$. As in case 0), it is away from zero.
	
	Case 3): Here $\gamma$ is strictly larger than another $\gamma$ of type 2), in at most the same dimension. Therefore $\gamma$ is away from zero.
	
	Case 4): Here $0=\sum_{i=0}^pz_ie_i$, unique with the property $z_i\in \Z$, $\gcd(z_i)=1$. Moreover, $z_0>0$. It follows that 
	$$
	\alpha_0=\frac{a_0z_0}{a_0z_0+\cdots+a_pz_p}.
	$$
	As above, $\alpha_0$ is away from zero. Since $\gamma/\alpha_0$ is another 
	$\gamma$ of type 1) in at most the same dimension, we deduce that $\gamma$ is away from zero.
\end{proof}

\begin{proof}[Proof of Theorem~\ref{GenPik}] We show by induction on $d$ that $\gamma(\Lambda,\square)\ge c(d,\cA)$, where $c(d,\cA)$ is a positive constant which depends only on $d$ and the ACC set $\cA$. If $d=1$, apply Lemma~\ref{gP1}. Suppose $d>1$. Denote $\gamma=\gamma(\Lambda,\square)$.
	
Denote $\tau=\lambda_1(\Lambda,\square-\square)$. Choose $0\ne b\in \Lambda\cap \tau (\square-\square)$. It must be primitive. Let $\pi\colon \Lambda\to \Lambda'=\Lambda/\Z b$ be the projection, let $\square'$ be the image of $\square$ under $\pi\colon \Lambda_\R\to \Lambda'_\R$. 
Since $\pi$ maps interior points of $\square$ onto interior points of $\square'$, we obtain $\Lambda'\cap \Int\square'\ne \emptyset$.
	
We have $0\in \square'$. Let $0\ne v'\in \square'(0)$. There exists $0\ne v\in \square(0)$ such that $\pi(v)=v'$. May write $v=\frac{e}{a_e}$, where $e\in \Lambda^{prim}$ and $a_e\in \cA$. Then $\pi(e)=z e'$, with $e'\in {\Lambda'}^{prim}$ and $z\in \Z_{\ge 1}$. Therefore $v'=e'/a_{e'}$ with $a_{e'}=\frac{a_e}{z}$. We deduce that $(\Lambda',\square')$ satisfies the same assumptions as $(\Lambda,\square)$, with the new ACC set $\cA'=\cA\cdot\{\frac{1}{z};z\in \Z_{\ge 1}\}$. By induction on $d$, $\gamma':=\gamma(\Lambda',\square')$ satisfies $\gamma'\ge c(d-1,\cA')>0$.

We claim that the following inequality holds:
$$
\tau+2\gamma\ge \gamma'.
$$ 
Indeed, we may suppose $\tau<\gamma'$. Choose $P'\in \Lambda'\cap \Int\square'$ such that $\gamma'=\gamma(P'\in \square')$. 
There exist $x,y\in \partial\square$ such that $x-y=\frac{b}{\tau}$.
Denote $\pi(x)=\pi(y)$ by $v'$. We have $P'\in (1-\gamma')\square'+\gamma' v'$. Then $P'= (1-\gamma')w'+\gamma' v'$ for some $w'\in \square'$. 
Therefore $\square_{P'}\supseteq (1-\gamma')\square_{w'}+\gamma'\square_{v'}$, where for $x'\in \square'$ we denote $\square_{x'}=\{x\in \square; \pi(x)=x'\}$. We obtain
	$\length \square_{P'} \ge (1-\gamma')\length\square_{w'}+\gamma'\length \square_{v'}$. Since $\square_{v'}=[x,y]$ is an interval of length $\tau^{-1}$, we obtain 
	$$
	\length \square_{P'}\ge \frac{\gamma'}{\tau}.
	$$
	Since $\tau<\gamma'$, it follows that $\square_{P'}$ is an interval of length strictly larger than $1$, hence it contains an interior lattice point $P$. From Lemma~\ref{lig}, we may choose $P$ so that 
	$$
	2\gamma(P\in \square_{P'})+\frac{1}{\length \square_{P'}}\ge 1.
	$$
	We obtain
	$$
	\gamma(P\in \square_{P'})\ge \frac{\gamma'-\tau}{2\gamma'}.
	$$
	By Lemma~\ref{cecr}, we deduce
	$$
	\gamma(P\in \square)\ge \gamma(P\in \square_{P'})\cdot \gamma(P'\in \square')\ge \frac{\gamma'-\tau}{2}.
	$$
	Therefore $\gamma\ge \frac{\gamma'-\tau}{2}$, which is equivalent to the claim.
	
	Suppose $\tau\le \frac{\gamma'}{2}$. Then $\gamma\ge \frac{\gamma'}{4}\ge \frac{c(d-1,\cA')}{4}$. Therefore $\gamma$ is bounded away from zero.
	
	Suppose $\tau\ge \frac{\gamma'}{2}$. Then $\tau \ge \frac{c(d-1,\cA')}{2}$, hence $\tau$ is bounded away from zero.
	Therefore $(\Lambda,\square)$ is bounded. Then $\gamma$ is bounded away from zero by Lemma~\ref{BndPik}.
\end{proof}

\begin{rem}
	Theorem~\ref{GenPik} implies the non-effective version of Pikhurko's Theorem~\cite{Pikh01}. Indeed, note first that Pikhurko's result takes the following equivalent form: if $\square\subset \Lambda^d_\R$ is a $d$-dimensional compact polytope with vertices
	contained in $l^{-1}\Lambda$, for a positive integer $l$, and $\Lambda\cap \Int\square\ne \emptyset$, then $\gamma(\Lambda,\square)\ge c(d,l)>0$.
	
	Choose an interior lattice point $P_0$. Denote $\square_0=\square-P_0$. Then $\square_0$ is a polytope containing $0$, its non-zero vertices are of the form $\frac{e}{a_e}$ where $e\in \Lambda^{prim}$ and $a_e$ belongs to the ACC set $\{\frac{l}{z};z\in \Z_{\ge 1}\}$, and there
	exists a bijection between the interior lattice points of $\square_0$ and those of $\square$, preserving the coefficient of asymmetry.
	Therefore $\gamma(\Lambda,\square)=\gamma(\Lambda,\square_0)\ge c(d,l)>0$.
\end{rem}

%%%%%%%%%%%%%%%%%%%%%%%%%%%%%%%%%%%%%%%%

\subsection{A generalization of Nill-Ziegler's theorem~\cite{NZ11} }

%%%%%%%%%%%%%%%%%%%%%%%%%%%%%%%%%%%%%%%%

\begin{thm}\label{SerPoly}
	Let $\Lambda\simeq \Z^d$ be a $d$-dimensional lattice, let 
	$\square\subset \Lambda_\R$ be a $d$-dimensional compact polytope
	such that $\Lambda\cap \Int\square=\emptyset$.
	Suppose $0\in \square$ and $\square(0)\subset \{0\}\cup \Lambda^{prim}/\cA$, where $\cA\subset (0,+\infty)$ is a given ACC set.
	Then there exists a non-zero projection of lattices 
	$\pi\colon \Lambda\to \Lambda'$ such that if we denote by $\square'=\pi_\R(\square)$ the image polytope, then $\Lambda'\cap \Int\square'=\emptyset$ and $(\Lambda',\square')$ is bounded.
\end{thm}

\begin{proof} We use induction on $d$.
	Suppose $d=1$. We may suppose $\Lambda=\Z$ and $\square$ is an interval with no interior lattice points. Equivalently, $\square\subseteq [z,z+1]$ for some $z\in \Z$. We may take the identity as the projection.
	
	Suppose from now that $d>1$ and the statement holds in smaller dimension.
	Denote $\tau=\lambda_1(\Lambda,\square-\square)$. 
	Choose $0\ne b\in \Lambda\cap \tau(\square-\square)$. Let $\Lambda\to \Lambda':=\Lambda/\Z b$ be the projection, let $\square'$ be the image of $\square$. We have $0\in \square'$ and $\square'(0)\subset {\Lambda'}^{prim}/\cA'$, where $\cA'=\cA\cdot \{\frac{1}{n};n\in \Z_{\ge 1} \}$ satisfies again ACC. We have two possibilities:
	
	1) Case $\Lambda'\cap \Int\square'=\emptyset$. Then we may replace $(\Lambda,\square,\cA)$ by $(\Lambda',\square',\cA')$. By induction,
	there exists a projection $\Lambda'\to \Lambda''$ with the desired properties. Its composition with $\Lambda\to \Lambda'$ satisfies the claim.
	
	2) Case $\Lambda'\cap \Int\square'\ne \emptyset$. We claim that the following inequality holds:
	$$
	\tau\ge \gamma(\Lambda',\square').
	$$
	Indeed, let $Q\in \Lambda'\cap \Int\square'$. Since $\square$ contains
	no interior lattice points, the fiber $\square_Q$ of $\square$ over $Q$ is an interval without interior integer points (with respect to basis $b$). Therefore
	$\length \square_Q\le 1$. On the other hand, there exists $p\in \square'$ such that $\length \square_p=\tau^{-1}$.
	If $Q\in (1-t)\square'+tp$, then $\length \square_Q\ge \frac{t}{\tau}$, and therefore $\tau\ge t$. We deduce 
	$$
	Q\notin \Int((1-\tau)\square'+\tau p).
	$$
	Let $Q'\in \partial (\square')$ be the unique point such that $Q\in (p,Q')$. Then $\frac{Q-p}{Q'-p}\ge 1-\tau$, that is 
	$$
	\tau\ge \frac{Q'-Q}{Q'-p}\ge \gamma(Q\in \square').
	$$
	Taking the maximum after all $Q$, we obtain the claim.
	
	By Theorem~\ref{GenPik}, $\gamma(\Lambda',\square')\ge \gamma(d-1,\cA')$. From the claim, we deduce that $\tau$ is bounded away from zero, i.e. $(\Lambda,\square)$ is bounded. Take the identity as the projection. 
\end{proof}

%\clearpage
%%%%%%%%%%%%%%%%%%%%%%%%%%%%%%%%%%%%%%%%
%%%%%%%%%%%%%%%%%%%%%%%%%%%%%%%%%%%%%%%%

\section{Germs of toric fibrations}

%%%%%%%%%%%%%%%%%%%%%%%%%%%%%%%%%%%%%%%%
%%%%%%%%%%%%%%%%%%%%%%%%%%%%%%%%%%%%%%%%

Let $(X/Y\ni P,B)$ consist of a {\em toric fibration} (i.e. a proper toric morphism with connected fibers) $f\colon X\to Y$,
where $Y$ is an affine toric variety with (unique) invariant point $P$, and
an $\R$-Weil divisor $0\le B\le \Sigma_X=X\setminus T_X$. 
We choose a canonical divisor $K$ on $X$ such that $K+\Sigma_X=0$. Note that $\mld_{f^{-1}P}(X,\Sigma_X)=0$.

In toric language, $f\colon T_N\emb(\Delta)\to T_{\bar{N}}\emb(\bar{\sigma})$ corresponds to a surjective homomorphism of lattices $\pi\colon N\to \bar{N}$, 
a strongly convex rationally polyhedral cone $\bar{\sigma}\subset \bar{N}_\R$ and a finite $N$-rational fan $\Delta$ with support $|\Delta|=\pi^{-1}(\bar{\sigma})$.
We have an induced injective homomorphism between dual lattices $\pi^*\colon \bar{M}\to M$. The existence of an invariant point on $Y$  is equivalent to $\bar{\sigma}-\bar{\sigma}=\bar{N}_\R$.

We use the same notation for a homomorphism of lattices and its extension to a homomorphism of $\R$-vector spaces.
Note that $|\Delta|$ contains $\Ker\pi$. Therefore 
$|\Delta|^\vee=\pi^*(\bar{\sigma}^\vee)$, and the correspondence $\bar{\tau}\mapsto \pi^{-1}(\bar{\tau})$ is a bijection between the faces of the cones $\bar{\sigma}\subset \bar{N}_\R$ and $|\Delta|\subset N_\R$. In particular, 
$$
\Int|\Delta|=\pi^{-1}(\Int\bar{\sigma}), \  \partial|\Delta|=\pi^{-1}(\partial\bar{\sigma}).
$$
If we denote $\Delta(1)=\{e_i;i\}$, then $B=\sum_i (1-a_i)V(e_i)$ with $a_i\in [0,1]$. Note that $(e_i)_i$ generate $N_\R$, since $\bar{\sigma}(1)$ generates $\bar{N}_\R$.
The invariant $\R$-Weil divisor $K+B=-\sum_ia_iV(e_i)$ may not be $\R$-Cartier. Its {\em moment polytope} is 
$$
\square_{-K-B}=\{m\in M_\R; (\chi^m)-K-B\ge 0\}=
\{m\in M_\R; \langle m,e_i\rangle+a_i\ge 0 \ \forall i\}.
$$
Denote 
$$
U=\Conv(\{0\}\cup \{\frac{e_i}{a_i};a_i>0\})+\sum_{a_i=0}\R_{\ge 0}e_i\subset N_\R.
$$
Both $\square$ and $U$ are polyhedral convex sets containing the origin. Moreover, $a_i=0$ for all $i$ if and only if $B=\Sigma_X$. 

We have $\square=C+\pi^*(\bar{\sigma}^\vee)$, where $C$ is a compact polytope. In particular, $\square$ is compact if and only if $Y=P$. Note that $U$ is compact if and only if $a_i>0$ for all $i$.

The pair $(M,\square)$ determines $\pi,|\Delta|,\bar{\sigma}$. Indeed,
$|\Delta|=\{e\in N_\R; \inf \langle e,\square\rangle>-\infty\}$. Since
$0$ is a vertex of $\bar{\sigma}$, we have 
$\Ker\pi=\{e\in N_\R; \length \langle e,\square\rangle<+\infty\}$. Therefore the quotient $N_\R\to \bar{N}_\R$ can be recovered from $(M,\square)$, and $\bar{N},\bar{\sigma}$ are the images of $N,|\Delta|$.

By definition, $\square=U^*$. By duality for polar sets, $U=\square^*$, i.e.
$$
U=\{e\in N_\R; -h_\square(e)\le 1\}.
$$ 
Therefore $|\Delta|=\cup_{t\ge 0}tU$.

\begin{lem}\label{inde}
$\Int U =\{e\in \Int|\Delta|; -h_\square(e)<1\}$.
\end{lem}

\begin{proof} The inclusion $\subseteq$ is clear. For the opposite inclusion, suffices to fix $e$ with $-h_\square(e)<1$ and show that $e\in \partial U$ if and only if $e\in \partial |\Delta|$. 
	
Note that $e\in U$. We have $e\in \partial U$ if and only if there exists $m\in M_\R\setminus 0$ such that 
$$
\min\langle m,U\rangle=\langle m,e\rangle=:c.
$$
There exists $t\in (0,1)$ such that $-h_\square(e)\le t$. That is $e\in tU$. Therefore $\langle m,t^{-1}e\rangle\ge c$,
i.e. $t^{-1}c\ge c$. Therefore $c\ge 0$. But $c\le 0$ since $0\in U$.
Therefore $c=0$. We obtain 
$$
\min\langle m,U\rangle=\langle m,e\rangle=0.
$$
Since $|\Delta|=\cup_{t\ge 0}tU$, the condition $\min\langle m,U\rangle=0$ is
equivalent to $m\in |\Delta|^\vee$. Therefore $e\in \partial U$ if and only if there exists $0\ne m\in |\Delta|^\vee$ such that $\langle m,e\rangle=0$, i.e. $e\in \partial |\Delta|$.
\end{proof}

Note that $U\cap \partial |\Delta|$ is the union of proper faces of $U$ which contain the origin. Also, for $t>0$, we have 
$-h_\square(e)\ge t\ \forall e\in N^{prim}\cap \Int |\Delta|$ if and only if $N^{prim}\cap \Int(tU)=\emptyset$, if and only if $N\cap \Int(tU)\subset \{0\}$. Therefore 
$$
\min\{-h_\square(e);e\in N^{prim}\cap \Int|\Delta| \}=
\max\{t\ge 0; N\cap \Int(tU)\subset \{0\}\}.
$$

\begin{lem}\label{semimld}
Suppose $-K-B$ is $f$-semiample. Let $t>0$. Then $\mld_{f^{-1}P}(X,B)\ge t$ if and only if $N\cap \Int(tU)\subset \{0\}$.
\end{lem}

\begin{proof}
In particular, $K+B$ is $\R$-Cartier. For each maximal cone $\sigma\in \Delta(top)$, there exists $\psi_\sigma\in M_\R$ such that 
$(\chi^{\psi_\sigma})+K+B=0$ on the affine open subset $U_\sigma\subset X$.
Each $e\in N^{prim}\cap \sigma$ defines an invariant prime divisor
$E_e$ on a toric resolution of $U_\sigma$, with log discrepancy 
$$
a_{E_e}(X,B)=\langle \psi_\sigma,e\rangle.
$$
Note that $\square=\cap_{\sigma\in \Delta(top)}(-\psi_\sigma+\sigma^\vee)$.
Since $Y$ is affine, the assumption $-K-B$ $f$-semiample is equivalent to 
$-\psi_\sigma\in \square$ for every $\sigma\in \Delta(top)$. For $e\in N^{prim}\cap \sigma$, we obtain
$$
a_{E_e}(X,B)=-\langle -\psi_\sigma,e\rangle=-\min\langle \square,e\rangle.
$$
Since $U_\sigma$ cover of $X$, we obtain
$$
a_{E_e}(X,B)=-h_\square(e) \ \forall e\in N^{prim}\cap |\Delta|.
$$

Note that $c_X(E_e)\subset f^{-1}P$ if and only if $\pi(e)\in \Int\bar{\sigma}$, which is equivalent to $e\in \Int|\Delta|$. Also, $(X,B)$ admits a toric log resolution such that the preimage of $P$ is an invariant subset, either the whole ambient, or of pure codimension one.
Therefore $\mld_{f^{-1}P}(X,B)\ge t$ if and only if 
$-h_\square(e)\ge t$ for every $e\in N^{prim}\cap \Int|\Delta|$.
From above, this is equivalent to $N\cap \Int(tU)\subset \{0\}$.
\end{proof}

We have $0\in \Int U$ if and only if $Y=P$, by Lemma~\ref{inde}. Therefore we have two possibilities:
\begin{itemize}
	\item Suppose $Y=P$. Then $N\cap \Int(tU)\subset \{0\}$ if and only if $N\cap \Int(tU)=\{0\}$.
	\item Suppose $\dim Y>0$. Then $N\cap \Int(tU)\subset \{0\}$ if and only if $N\cap \Int(tU)=\emptyset$.
\end{itemize}

\begin{defn}
Let $t\ge 0$. We say that $(X/Y\ni P,B)$ satisfies {\em property 
$(C_t)$} if there exists an $\R$-Weil divisor $B^+\ge B$ such that $K+B^+\sim_\R 0$ and $\mld_{f^{-1}P}(X,B^+)\ge t$. 
\end{defn}

Note that $(X/Y\ni P,B)$ always satisfies property $(C_0)$, as we may take $B^+=\Sigma_X$.

\begin{lem}\label{chRco} Let $t>0$. 
\begin{itemize}
	\item[1)] Suppose $-K-B$ is $f$-semiample. If $(C_t)$ holds, then  $\mld_{f^{-1}P}(X,B)\ge t$. The converse holds except the following special case: $t=1$, $Y=P$, $B=0$, $X$ has canonical singularities and $-K$ is semiample.
	\item[2)] We have $-K-B=\sum_i a_iV(e_i) \ge 0$. Consider a commutative diagram	
\[ 
\xymatrix{
	   &   \tilde{X}\ar[dl]_\mu\ar@{-->}[dr]^\chi  &  \\
	X \ar[dr]_f &       &  X' \ar[dl]^{f'}   \\
	   &   Y   &       
}
\]
where $\mu$ is a toric $\Q$-factorialization, $\tilde{B}:=\mu^{-1}_*B$,
$\chi$ is an MMP over $Y$ for $-K_{\tilde{X}}-\tilde{B}$, $B':=\chi_*\tilde{B}$ and $-K_{X'}-B'$ is $f'$-semiample.
Then:
\begin{itemize}
	\item[a)] We have $\square(X,B)=\square(X',B')$
	and $U(X,B)=U(X',B')$.
	\item[b)] $(C_t)$ holds for $(X/Y\ni P,B)$ if and only if it holds for 
	$(X'/Y\ni P,B')$.
\end{itemize} 
\item[3)] Suppose $B$ has rational coefficients. If $(C_t)$ holds, there exists a solution $B^+$ with rational coefficients.
\item[4)] Suppose $(X'/Y\ni P,B')$ is not the special case in 1). Then 
$(C_t)$ holds for $(X/Y\ni P,B)$ if and only if $N\cap \Int(tU)\subset \{0\}$.
\end{itemize}

\end{lem}

\begin{proof} 1) Since $B\le B^+$, we have $\mld_{f^{-1}P}(X,B)\ge\mld_{f^{-1}P}(X,B^+)\ge t$. For the converse, suppose $\mld_{f^{-1}P}(X,B)\ge t$. Since $Y$ is affine, the assumption that $-K-B$ is $f$-semiample is equivalent to $K+B+\sum_ld_lC_l=0$ where 
$d_l$ are non-negative real numbers and $|C_l|$ are base point free linear systems. Let $n\ge \max_ld_l $. Choose general members $S_l\in |nC_l|$. Then $K+B+\sum_l \frac{d_l}{n}S_l\sim_\R 0$ and $\mld(X,B+\sum_l \frac{d_l}{n}S_l)\ge 0$.

Suppose $\dim Y>0$, i.e $f^{-1}P\subsetneq X$. Then no component of $S_l$
is contained in $f^{-1}P$. Therefore $\mld_{f^{-1}P}(X,B+\sum_l \frac{d_l}{n}S_l)\ge t$.

Suppose $\dim Y=0$, i.e. $f^{-1}P=X$. From $\mld(X,B)\ge t$ we deduce
$t\le 1$. Suppose $t<1$. Then we may choose $n$ such that $\max_l d_l\le n(1-t)$. Therefore $\mld_{f^{-1}P}(X,B+\sum_l \frac{d_l}{n}S_l)\ge t$.
Suppose $t=1$. Then $B=0$, $X$ has canonical singularities and $-K$ is semiample. If $K\le K+B^+\sim_\R 0$, then $B^+\ne 0$ since $-K$ is big on a proper toric variety. Therefore $(C_1)$ fails if $\dim Y=0$. 

2) The invariant $\R$-Weil divisor $-K_{\tilde{X}}-\tilde{B}=\sum_i a_i \tilde{V}(e_i)\ge 0$ is $\R$-Cartier, since $\tilde{X}$ is $\Q$-factorial.
Therefore its toric MMP over $Y$ exists, and terminates with a toric minimal model over $Y$. Since $\mu$ is small, we have 
$\square_{-K-B}=\square_{-K_{\tilde{X}}-\tilde{B} }$. We have 
$\overline{-K_{\tilde{X}}-\tilde{B} }=\overline{-K_{X'}-B_{X'}}+\bF$, 
where $\bF$ is an effective b-divisor which is exceptional over $X'$.
Therefore $\square_{-K_{\tilde{X}}-\tilde{B} }=\square_{-K_{X'}-B'}$.
We obtained $\square_{-K-B}=\square_{-K_{X'}-B'}$. By duality, 
$U(X,B)=U(X',B')$. This proves a).

For b), note first that property $(C_t)$ for $K+B$ is equivalent to that for $K_{\tilde{X}}+\tilde{B}$ (take push-forwards), since $\mu$ is small. 
Property $(C_t)$ for $K_{\tilde{X}}+\tilde{B}$ implies the same for $K_{X'}+B'$ (take push-forwards). The converse holds as well:  
${B'}^+$ is a solution on $X'$, then $(\overline{{B'}^+}+\bF)_X$ is a 
solution on $X$, since $\bF$ is effective. We conclude that b) holds.

3) Suppose $B$ has rational coefficients. Then so do $\tilde{B}$ and $B'$. Therefore $\bF$ has rational coefficients, hence there exists a rational solution on $X$ if and only if there exists a rational solution on $X'$.
We are reduced to the case $-K-B$ is $f$-semiample. Since $B$ has rational coefficients, this is equivalent to $|-nK-nB|$ is base point free for some 
$n\ge 1$. The argument of 1) gives a rational solution.

4) Apply 2), 1) and Lemma~\ref{semimld}.
\end{proof}

Suppose $B$ has rational coefficients. By Lemma~\ref{chRco}.3) we deduce that $(C_t)$ holds if and only if there exists an integer $n\ge 1$ and a member $D\in |-nK-nB|$ such that $\mld_{f^{-1}P}(X,B+\frac{1}{n}D)\ge t$.

\begin{rem}
$(C_t)$ is in fact a property of the germ of the fibration near the special fiber. Suppose $P\in V\subseteq Y$ is an open neighborhood, and 
there exists $B^+_{f^{-1}V}\ge B|_{f^{-1}V}$ on $f^{-1}V$ such that 
$K_{f^{-1}V}+ B^+_{f^{-1}V}\sim_\R 0$ and $\mld_{f^{-1}P}(f^{-1}V,B^+_{f^{-1}V})\ge t$. We claim that $(C_t)$ holds
over $Y$.
Indeed, we may suppose $V\ne Y$. In particular, $Y\ne P$. Then the special
case does not appear in Lemma~\ref{chRco}.1), and we may suppose $-K-B$ is $f$-semiample. Since $B^+_{f^{-1}V}\ge B|_{f^{-1}V}$ and the geometric valuations of $X$ with center contained in $f^{-1}P$ are also valuations of $f^{-1}V$, we obtain $\mld_{f^{-1}P}(X,B)\ge t$. The argument of 
Lemma~\ref{chRco}.1) shows that $(C_t)$ holds.
\end{rem}

\begin{thm}\label{toricBAB} Fix $d\ge 1$ and $t>0$.
Let $X$ be a proper toric variety of dimension $d$. Suppose there exists $B^+\ge 0$ such that $K+B^+\sim_\R 0$ and $\mld(X,B^+)\ge t$. Then $X$ belongs to finitely many isomorphism types.
\end{thm}

\begin{proof}
We are in the above setup, with $Y$ a point and $B=0$. 
We may suppose $t<1$. Then the special case of Lemma~\ref{chRco} does not appear, and the assumption is equivalent to $N\cap \Int(tU)=\{0\}$. Note that $U=\Conv\Delta(1)$ is a polytope with extremal points in the lattice $N$, and $U$ contains $0$ in its interior. 
The argument of~\cite[Theorem 6.4]{Amb16} gives the claim: $\gamma(0\in U)$ is 
bounded away from zero, hence Minkowski's first main theorem gives 
$\vol_N(U-U)\le v(d,t)$. Since $U$ has extremal points in the lattice, it follows that the pair $(N,U)$ is finite up to isomorphisms of the lattice.
There are at most finitely many fans with the same set of $1$-dimensional cones. Therefore $(N,\Delta)$ is finite, up to
automorphisms of the lattice.
\end{proof}

Theorem~\ref{toricBAB} generalizes ~\cite{BB92} and ~\cite[Theorem 6.4]{Amb16}, and is the toric version of the following statement:
if $X$ is a proper variety of dimension $d$ such that there exists $B^+$ effective and big with $K+B^+\sim_\R 0$ and $\mld(X,B^+)\ge t>0$, 
then $X$ belongs to a bounded family depending only on $d$ and $t$. Indeed, by~\cite{Bir21} we may find $B\ge 0$ such that $n(K+B)\sim 0$ and 
$(X,B)$ is klt, where $n$ depends only on $d$ and $t$. By~\cite{HX15}, $X$ is bounded.

For the rest of this section, we suppose $\dim Y>0$. Then the special case of Lemma~\ref{chRco}.1) does not appear, $0\in \partial U$ and 
$(X/Y\ni P,B)$ satisfies $(C_t)$ if and only if $N\cap \Int(tU)=\emptyset$.

\begin{thm}\label{bndser} Fix $d\ge 1$ and an ACC set $\cA\subset (0,+\infty)$. 
Let $(X/Y\ni P,B)$ be a data as above, let $t>0$. Suppose $\dim X=d$, $\dim Y>0$ and $\frac{a_i}{t}\in \cA$ for all $e_i\in \Delta(1)$. 
The following are equivalent:
\begin{itemize}
 \item[a)] $(X/Y\ni P,B)$ satisfies $(C_t)$.
 \item[b)] There exists a non-zero surjective homomorphism of lattices $\Phi\colon N\to N'$ such that the image $U'=\Phi(U)\subset N'_\R$
 is a compact polytope, $N'\cap \Int(tU')=\emptyset$ and $(N',tU')$ is bounded.
\end{itemize}
\end{thm}

\begin{proof} $b)\Longrightarrow a)$: since $\Phi$ maps $\Int(tU)$ into $\Int(tU')$, b) implies $N\cap \Int(tU)=\emptyset$, i.e. a) holds.

$a)\Longrightarrow b)$: 
Denote $\sigma_0=\sum_{a_i=0}\R_{\ge 0}e_i$. Let $g\colon N\to N_0:=N/N\cap (\sigma_0-\sigma_0)$ be the projection, let $U_0\subset N_{0,\R}$ 
be the image of $U$. We claim that 
$$
g(N\cap \Int(tU))=N_0\cap \Int(tU_0).
$$
Indeed, the inclusion $\subseteq$ is clear. For the opposite inclusion, let $e_0\in N_0\cap \Int(tU_0)$.
There exist $e\in N$ and $v\in \Int(tU)$ with $g(e)=g(v)=e_0$. Then $e-v=v_2-v_1$ for some $v_1,v_2\in \sigma_0$.
We may write $v_1=\sum_{a_i=0}t_ie_i$ for some $t_i\ge 0$. Then 
$$
e+\sum_{a_i=0}\lceil t_i\rceil e_i = v+(v_2+\sum_{a_i=0}(\lceil t_i\rceil-t_i)e_i)\subset v+\sigma_0.
$$
Since $v+\sigma_0\subset \Int(tU)$, the lattice vector $e+\sum_{a_i=0}\lceil t_i\rceil e_i$ belongs to $\Int(tU)$, and maps into $e_0$.

Our assumption is $N\cap \Int(tU)=\emptyset$. Therefore $N_0\cap \Int(tU_0)=\emptyset$.
The image $U_0=\Conv(\{0\}\cup \{\frac{g(e_i)}{a_i};a_i>0 \})$ is a compact polytope. Note that $a_i>0$ for at least one $i$, since $t>0$.
Suppose $a_i>0$ and $g(e_i)\ne 0$. Then $g(e_i)=qe_0$ for some $q\in \Z_{\ge 1}$ and $e_0\in N_0^{prim}$. Therefore 
$$
\frac{g(e_i)}{a_i}=\frac{e_0}{a_0},\ a_0=\frac{a_i}{q}.
$$
Therefore $\frac{a_0}{t}$ belongs to the ACC set $\cA':=\cA\cdot \{ \frac{1}{q};q\in \Z_{\ge 1} \}$. 

Thus $(N_0,tU_0)$ satisfies the assumptions of Theorem~\ref{SerPoly}.
Therefore there exists a non-zero surjective homomorphism of lattices
$\Phi_0\colon N_0\to N'$ such that the image 
$U'=\Phi_0(U_0)\subset N'_\R$ satisfies the properties in b). The composition $\Phi=\Phi_0\circ g\colon N\to N'$ satisfies b).
\end{proof}

We call $\Phi\colon (N,tU)\to (N',tU')$ as in b) a {\em $t$-lc reduction} of $(X/Y\ni P,B)$. Assuming it exists,  we analyze it.
Note that $\Phi$ must factor through $N_0$, since the image of $U$ is compact. 

We have $0\in \partial U'$. Indeed, we have $0\in U'$. If $0\in \Int U'$, then $0\in N'\cap \Int(tU')$. Contradiction!

Denote $\sigma'=\cup_{t\ge 0}tU'\subset N'_\R$. Then $\pi(|\Delta|)=\sigma'$ and 
$$
{\sigma'}^\vee = \{m'\in M'_\R ; \Phi^*(m')\in |\Delta|^\vee \} 
 = \{m'\in M'_\R; \exists \bar{m}\in \bar{\sigma}^\vee, \Phi^*(m')=\pi^*(\bar{m})\}.
$$
We compute
$$
U'=\Conv(\{0\}\cup \{\frac{\Phi(e_i)}{a_i};\Phi(e_i)\ne 0 \})\subset N'_\R.
$$
Suppose $\Phi(e_i)\ne 0$. In particular, $a_i>0$. Write uniquely $\Phi(e_i)=q_ie'_i$  for some $q_i\in \Z_{\ge 1}$ and $e'_i\in {N'}^{prim}$. Therefore 
$$
\frac{\Phi(e_i)}{a_i}=\frac{e'_i}{a'_i},\ a'_i=\frac{a_i}{q_i}.
$$
Note that $\frac{a'_i}{t}$ belongs to the ACC set $\cA':=\cA\cdot \{ \frac{1}{q};q\in \Z_{\ge 1} \}$. 

\begin{lem}
The set $\{e'_i;\Phi(e_i)\ne 0\}$ is finite up to $\Aut(N')$, $q_i$ is finite, and $\frac{a'_i}{t}$ is bounded away from zero.
\end{lem}

\begin{proof}
After choosing a basis we may suppose $N'=\Z^{d'}$ and $tU'$ is 
contained in $[-l,l]^{d'}$ for some $l=l(d,\cA)$. This box contains $t\frac{\Phi(e_i)}{a_i}=(\frac{a'_i}{t})^{-1}e'_i$.
Since $\frac{a'_i}{t}$ satisfies ACC, it is bounded above. Therefore $e'_i$ belongs to a finite set, and $\frac{a'_i}{t}$ is bounded away from zero. The latter implies that $q_i$ belongs to a finite set.
\end{proof}

Note that every non-zero extremal point of $U'$ is the $\Phi$-image of a non-zero extremal point of $U$, and the latter belongs to the finite set 
$\{\frac{e_i}{a_i};\Phi(e_i)\ne 0\}$. In particular, any non-zero extremal point of $tU'$ is of the form $(\frac{a'_i}{t})^{-1}e'_i$, with the notations above.

\begin{lem}
The pair $(N',\sigma')$ belongs to finitely many isomorphism types.
\end{lem}

\begin{proof} We have $\sigma'=\sum_{\Phi(e_i)\ne 0}\R_{\ge 0}e'_i$. 
With respect to some basis of $N^*$, the vectors $e'_i$ have finitely many coefficients. It follows that $(N',\sigma')$ is finite up to $\Aut(N')$.
\end{proof}

\begin{lem}\label{finser} 
In the setting of Theorem~\ref{bndser}, suppose $\cA$ may only accumulate to $0$. Then $(N',tU')$ belongs to finitely many isomorphism types.
\end{lem}

\begin{proof}
Since $\frac{a'}{t}\in \cA$ is bounded away from zero, it must be finite.
Then the extremal points of $tU'$ belong to a finite set, modulo $\Aut(N')$.
Therefore $tU'$ is finite up to $\Aut(N')$.
\end{proof}

Let $\Phi^*\colon M'\to M$ be the induced injective homomorphism between dual lattices. Define
\begin{align*}
	\square' & = \{m'\in M'_\R; \Phi^*(m')\in \square\} \\
	         & = \{m'\in M'_\R; \langle m',\Phi(e_i)\rangle+a_i\ge 0\ \forall \Phi(e_i)\ne 0 \}.
\end{align*}
Note that $\Phi(e_i)\ne 0$ for some $i$, since $\Phi\ne 0$ and $e_i$ generate $N_\R$.
By construction, $\square'={U'}^*$. By duality, $U'={\square'}^*$, i.e.
$$
U'=\{e'\in N'_\R; -h_{\square'}(e')\le 1 \}=\{e'\in \sigma'; -h_{\square'}(e')\le 1 \}.
$$
For $-h_{\square'}(e')<1$, we have $e'\in \partial U'$ if and only if $e'\in \partial \sigma'$. Therefore 
$$
\Int U'=\{e'\in \Int \sigma'; -h_{\square'}(e')< 1 \}.
$$
Therefore the property $N'\cap \Int(tU')=\emptyset$ is equivalent to 
$$
\min\{-h_{\square'}(e'); e'\in N'\cap \Int\sigma'\}\ge t.
$$

\begin{lem}\label{cbe}
$-\min_{m'\in \square'}\langle \Phi^*(m'),e\rangle \ge t$ for every 
$e\in N\cap \Int |\Delta|$.
\end{lem}

\begin{proof}
Since $\langle \Phi^*(m'),e\rangle=\langle m',\Phi(e)\rangle$ and $\Phi(N\cap \Int|\Delta|)\subset N'\cap \Int\sigma'$.
\end{proof}

\begin{thm}\label{bndmult}
Let $f\colon (X,B=\sum_i(1-a_i)E_i)\to Y\ni P$ be the germ of a toric fibration near an invariant point. Suppose $\dim Y>0$, $-K-B$ is $f$-semiample, $\mld_{f^{-1}P}(X,B)\ge t>0$ and $\frac{a_i}{t}$ belong to an ACC set $\cA$. Then there exists an invariant hyperplane section
$H\in |\fm_{Y,P}|$ such that $\mld(X,B+\gamma f^*H)\ge 0$ for some
$\gamma>0$ with $\frac{\gamma}{t}$ bounded away from zero.
\end{thm}

\begin{proof}
Let $\Phi\colon (N,tU)\to (N',tU')$ be a $t$-lc reduction from Theorem~\ref{bndser}. Choose $m'\in {\sigma'}^\vee(1)$. Then $\max\langle m',tU'\rangle$ is a positive real number, since $U'$ generates $\sigma'$ and $m'\in {\sigma'}^\vee$ is not zero. The maximum is bounded above, since in some basis of $N'$ both $tU'$ and ${\sigma'}^\vee(1)$ (with respect to dual basis) are contained in a bounded box.

There exists a unique $0\ne \bar{m}\in \bar{M}\cap \bar{\sigma}^\vee$
such that $\pi^*(\bar{m})=\Phi^*(m')$. Let $H:=(\chi^{\bar{m}})\in |\fm_{Y,P}|$. Then $K+B+\gamma f^*H$ has log canonical singularities if and only if $\gamma\langle \pi^*(\bar{m}),e_i\rangle\le a_i$ for all $i$,
if and only if $\gamma\langle \Phi^*(m'),e_i\rangle\le a_i$ for all $i$,
if and only if $t/\gamma\ge \max\langle m',tU'\rangle$. Therefore the maximal value for $\gamma$ is 
$$
\gamma=\frac{t}{\max\langle m',tU'\rangle}.
$$
From above, $\frac{\gamma}{t}$ is bounded away from zero.

Note that we can find $p$ linearly independent such hyperplane sections,
where $p=\dim N'$.
\end{proof}

In particular, if $\dim Y=1$ and $t$ is bounded away from zero,
then the coefficients of the special fiber $f^*P$ are bounded
above (cf.~\cite[Corollary 1.5]{BC21}).

%\clearpage
%%%%%%%%%%%%%%%%%%%%%%%%%%%%%%%%%%%%%%%%
%%%%%%%%%%%%%%%%%%%%%%%%%%%%%%%%%%%%%%%%

\section{Complements of bounded index}

%%%%%%%%%%%%%%%%%%%%%%%%%%%%%%%%%%%%%%%%
%%%%%%%%%%%%%%%%%%%%%%%%%%%%%%%%%%%%%%%%

Let $r$ be a positive integer. Denote $\cA_r=\{\frac{x}{q};
x\in [0,1],rx\in \Z,q\in \Z_{\ge 1}\}$. The set $\{1-a;a\in \cA_r\}$ is called the set of {\em $r$-hyperstandard coefficients}.

\begin{thm}\label{bndscoP}
Let $f\colon X\to Y$ be a toric fibration, with $\dim X=d$ and $Y$ affine with (unique) invariant point $P$. Suppose $\dim Y>0$.
Let $0\le B\le \Sigma_X$ have
$r$-hyperstandard coefficients. Let $t>0$. The following are equivalent:
\begin{itemize}
	\item[1)] There exists an open neighborhood $P\in V\subseteq Y$ and 
	$B^+_{f^{-1}V}\ge B|_{f^{-1}V}$ on $f^{-1}V$ such that $K_{f^{-1}V}+ B^+_{f^{-1}V}\sim_\R 0$ and $\mld_{f^{-1}P}(f^{-1}V,B^+_{f^{-1}V})\ge t$.
	\item[2)] There exists $B^+\ge B$ such that $n(K+B^+)\sim 0$ and 
	$\mld_{f^{-1}P}(X,B^+)\ge t$, where $n$ is a positive integer which depends only on $d,r,t$.
\end{itemize}
\end{thm}

\begin{proof} $1)\Longrightarrow 2)$:
Let $f\colon T_N\emb(\Delta)\to T_{\bar{N}}\emb(\bar{\sigma})$ be our morphism. We apply Theorem~\ref{bndser} with the ACC set $\cA=t^{-1}\cA_r$.
Let $\Phi\colon (N,tU)\to (N',tU')$ be the induced non-zero projection.
Since $0$ is the only cluster point of $\cA$, Lemma~\ref{finser} shows that $(N',tU')$ is finite up to isomorphism. Since $t$ is fixed,
$(N',U')$ is finite up to isomorphism. Therefore $(M',\square')$ is finite up to isomorphism. Since $U'$ has rational extremal points, so does $\square'$. Therefore there exists $n=n(d,r,t)\ge 1$ such that the polyhedral convex set $\square'\subset M'_\R$ has extremal points contained in $\frac{1}{n}M'$. Let $A'\subset M'$ be the finite set of extremal points of $n\square'$.

We have $\Phi^*(A')\subset M\cap n\square$. The characters $\{ \chi^{\Phi^*(a')};a'\in A'\}\subset \Gamma(X,-nK-nB)$
define an invariant linear system $\Lambda\subset |-nK-nB|$. Let $D\in |-nK-nB|$ be a general member. We have $n(K+B+\frac{1}{n}D)\sim 0$.
For a toric valuation $E_e$ over $X$ induced by some $e\in N^{prim}\cap |\Delta|$, we compute the log discrepancy
$$
a_{E_e}(X,B+\frac{1}{n}D)=-\frac{1}{n}\min_{a'\in A'}\langle \Phi^*(a'),e \rangle.
$$
Now $\langle \Phi^*(a'),e \rangle=\langle a',\Phi(e) \rangle$. Since
$\frac{a'}{n}$ are the extremal points of $\square'$, we obtain
$$
a_{E_e}(X,B+\frac{1}{n}D)=-h_{\square'}(\Phi(e)).
$$

Suppose $e\in N^{prim}\cap \Int|\Delta|$. By Lemma~\ref{cbe}, 
$-h_{\square'}(\Phi(e))\ge t$. Therefore $a_{E_e}(X,B+\frac{1}{n}D)\ge t$.
Since $(X,B+\frac{1}{n}D)$ admits a toric log resolution, we conclude 
that $\mld_{f^{-1}P}(X,B+\frac{1}{n}D)\ge t$.
\end{proof}

\begin{thm}\label{bndscogl}
	Let $f\colon X\to Y$ be a toric fibration, with $\dim X=d$ and $Y$ affine. Let $0\le B\le \Sigma_X$ have
	$r$-hyperstandard coefficients. Let $t\in (0,1)$. The following are equivalent:
	\begin{itemize}
		\item[1)] For every closed point $y\in Y$, there exists an open neighborhood $y\in V\subseteq Y$ and 
		$B^+_{f^{-1}V}\ge B|_{f^{-1}V}$ on $f^{-1}V$ such that $K_{f^{-1}V}+ B^+_{f^{-1}V}\sim_\R 0$ and $\mld(f^{-1}V,B^+_{f^{-1}V})\ge t$.
		\item[2)] There exists $B^+\ge B$ such that $n(K+B^+)\sim 0$ and 
		$\mld(X,B^+)\ge t$, where $n$ is a positive integer which depends only on $d,r,t$.
	\end{itemize}
\end{thm}

\begin{proof} $1)\Longrightarrow 2)$: 
 We have $\Gamma(-nK-nB)=\oplus_{m\in M\cap n\square}k\cdot \chi^m$. Let $A\subset M\cap n\square$ be a finite set such that $\{\chi^a;a\in A \}$ generate $\Gamma(-nK-nB)$ as a $\Gamma(\cO_Y)$-module (equivalently, $M\cap n\square=\cup_{a\in A} a+\pi^*(\bar{M}\cap \bar{\sigma}^\vee))$. The invariant subspace $\oplus_{a\in A}k \cdot \chi^a\subset \Gamma(-nK-nB)$ defines an invariant finite dimensional linear system $\Lambda\subset |-nK-nB|$. Let $D\in \Lambda$ be a general member, denote $B_n=B+\frac{1}{n}D$. By construction,
$B_n\ge B$ and $n(K+B_n)\sim 0$. For a toric valuation $E_e$ of $X$ induced by some 
$e\in N^{prim}\cap |\Delta|$, we compute the log discrepancy
\begin{align*}
a_{E_e}(X,B_n) & = -\frac{1}{n}\min\langle A,e\rangle \\
               & = -\frac{1}{n}\min\langle M\cap n\square,e\rangle.
\end{align*}
Since $(X,B_n)$ admits a toric resolution, we obtain that $\mld(X,B_n)\ge t$ if and only if $\frac{1}{n}\le 1-t$ and 
$-\frac{1}{n}\min\langle M\cap n\square,e\rangle\ge t$ for all $e\in N^{prim}\cap |\Delta|$. We abused notation since $B_n$ does depend on the choice of $A$, but we are only interested when $\mld(X,B_n)\ge t$, which does not depend on choices.

From $0\in \square$ we obtain $M\cap n\square\subset M\cap (n+1)\square$. Thus $\mld(X,B_n)\ge t$ implies $\mld(X,B_{n+1})\ge t$.
Property $\mld(X,B_n)\ge t$ holds for $n\gg 0$, since $\square$ has rational extremal points and $-h_\square(e)\ge t$ for all $e\in N^{prim}\cap |\Delta|$. The latter property is equivalent to 1), by the argument of Lemma~\ref{chRco}. 

a) If $Y$ does not have a closed invariant point, we may find $n=n(d-1,r,t)$ in 2). Indeed, suppose the unique closed orbit of $Y$ has positive dimension. Let $\bar{N}'=\bar{N}\cap (\bar\sigma-\bar\sigma)$.
The short exact sequence $0\to \bar{N}'\to \bar{N} \to Q\to 0$ splits.
Therefore $Y\simeq Y'\times T''$, where $T''=T_Q$ and $Y'$ is the affine $T_{\bar{N}'}$ embedding corresponding to $\bar\sigma$ viewed inside $\bar{N}'$. Note that $Y'$ admits an invariant closed point.
Let $N'=\pi^{-1}(\bar{N}')\subset N$. Let $X'$ be the equivariant  $T_{N'}$-embedding corresponding to the cones of $\Delta$ viewed as cones inside $N'$. We obtain a commutative diagram
\[ 
\xymatrix{
	  X \ar[d]_f \ar[rr]^\simeq  & &   X'\times T'' \ar[d]^{f'\times \id} \\
	  Y \ar[rr]^\simeq   & &   Y'\times T''    
}
\]
where $f'\colon X'\to Y'$ is a proper toric morphism, $Y'$ is affine with an invariant closed point, $K_X$ corresponds to $K_{X'}\times T''$
and $B$ corresponds to $B'\times T''$. We have $N\simeq N'\times Q$ and 
$N\cap |\Delta_X|$ corresponds to $(N'\cap |\Delta_{X'}|,0)$. Dually,
$M\simeq M'\times Q^*$ and $\square$ corresponds to $\square'\times Q^*$.
We deduce that 1) holds for $f'$. We have $\dim X'=\dim X-\dim T''<\dim X$.
By induction, $\mld(X',B'_n)\ge t$ for some $n=n(d-1,r,t)$. 
We obtain $\mld(X,B_n)\ge t$.

Therefore we may suppose $Y$ has an invariant closed point $P$.
	
b) Case $Y=P$. By Theorem~\ref{toricBAB}, the isomorphism class of $X$ is finite. Since the coefficients of $B$ may only cluster to $1$, they are finite. Therefore the pair $(M,\square)$ is finite up to isomorphism.
Then the compact polytope $n\square$ has extremal points in the lattice, for some $n=n(d,r,t)$. Let $D\in |-nK-nB|$ be a general member.
Then $n(K+B+\frac{1}{n}D)\sim 0$ and $\mld(X,B+\frac{1}{n}D)\ge t$.

c) Case $Y\ne P$. We claim that $K+B_n$ is $t$-lc over $X\setminus f^{-1}P$ for some $n=n(d-1,r,t)$. Indeed, suppose $e\in N^{prim}\cap |\Delta|$ induces a toric valuation $E_e$ of $X$ such that
$f(c_X(E_e))\ne P$ and $a_{E_e}(X,B_n)<t$. We have $\pi(e)\in \relint \bar{\tau}$ for some proper face $\bar{\tau} \prec \bar{\sigma}$. Let $Y^0$ be the invariant affine open subset $U_{\bar{\tau}}\subsetneq Y$. Let $X^0=f^{-1}(Y^0)$ and $f^0\colon X^0\to Y^0$ the induced toric proper morphism. The unique closed orbit of $Y^0$ is $O_{\bar{\tau}}$, of positive dimension. Note that $E_e$ is also a toric valuation over $X^0$.
Choose $\bar\varphi\in \bar{M} \cap \relint(\bar{\sigma}\cap \bar{\tau}^\perp)$. Then $Y^0$ coincides with the principal open set
$D(\chi^{\bar\varphi})\subset Y$. Therefore $m\in M$ belongs to 
$n\square^0=\square_{-nK_{X^0}-nB^0}$ if and only if 
$m+l\pi^*\bar\varphi$ belongs to $n\square=\square_{-nK-nB}$ for $l\gg 0$.
Since $\pi(e)\in \bar{\tau}$ we have $\langle \varphi,\pi(e)\rangle=0$,
hence $\langle m+l\pi^*\varphi,e\rangle=\langle m,e\rangle$ for all $l$. We deduce that $a_{E_e}(X,B_n)=a_{E_e}(X^0,B^0_n)$.
Therefore $a_{E_e}(X^0,B^0_n)<t$.
But $f^0$ is a toric fibration whose base does not admit an invariant point. By a), $n<n(d-1,r,t)$. This proves the claim.

By Theorem~\ref{bndscoP}, there exists $n_P=n(d,r,t)$ such that 
$\mld_{f^{-1}P}(X,B_{n_P})\ge t$. A valuation of $X$ is either defined over 
$X\setminus f^{-1}P$, or its center on $X$ is contained in $f^{-1}P$.
Therefore $\mld(X,B_{n'})\ge t$ for $n'\ge \max(n,n_P)$.
\end{proof}

\begin{exmp} If $d=1$, Theorems~\ref{bndscoP} and ~\ref{bndscogl} apply to only two possible cases:
	\begin{itemize}
		\item $(X,B)=(\bA^1,b\cdot 0)$ with $b=1-\frac{x}{q}$, $f=\id_{\bA^1}$ and $P=0$. The condition $\mld_P(X,B)\ge t$, i.e. $b\le 1-t$, implies $q\le t^{-1}$.
		Take $n=rq\le r\lfloor t^{-1}\rfloor$. Then $n(K+B)\sim 0$.
		\item $(X,B)=(\bP^1,b_0\cdot 0+b_\infty \cdot \infty)$ with $b_0=1-\frac{x_0}{q_0}$, $b_\infty=1-\frac{x_\infty}{q_\infty}$,
		and $f$ is the constant map to a point $P$. The condition $\mld(X,B)\ge t$ implies $q_0,q_\infty\le \lfloor t^{-1}\rfloor$. Let $l=\lcm(q_0,q_\infty)\le \lfloor t^{-1}\rfloor^2$. Let $n$ be a multiple of $rl$.
		Then $\Bs|-nK-nB|=\emptyset$. If moreover
		$n\ge (1-t)^{-1}$, the general member $D\in |-nK-nB|$ satisfies 
		$\mld(X,B+\frac{1}{n}D)\ge t$.  
	\end{itemize}	
\end{exmp}

%\clearpage
%%%%%%%%%%%%%%%%%%%%%%%%%%%%%%%%%%%%%%%%
%%%%%%%%%%%%%%%%%%%%%%%%%%%%%%%%%%%%%%%%

\section{Germs of toric singularities}

%%%%%%%%%%%%%%%%%%%%%%%%%%%%%%%%%%%%%%%%
%%%%%%%%%%%%%%%%%%%%%%%%%%%%%%%%%%%%%%%%

Let $P\in (X,B)$ be an affine toric log variety with invariant point $P$.
That is $X=T_N\emb(\sigma)$ is an affine equivariant $T_N$-embedding,
$\sigma(1)=\{e_i;i\}$ generates $N_\R$, $B=\sum_i(1-a_i)E_i$ with $a_i\in [0,1]$ and $K+B=-\sum_i a_iE_i$ is $\R$-Cartier. The latter property is equivalent to the existence of $\psi\in M_\R$ such that 
$(\chi^\psi)+K+B=0$, i.e. $\langle \psi,e_i\rangle=a_i$ for all $i$.
The linear form $\psi\in M_\R$ is unique, called the {\em log discrepancy function} of $(X,B)$, due to the following property: if $E_e$ is the toric valuation of $X$ corresponding to some $e\in N^{prim}\cap \sigma$, the log discrepancy of $(X,B)$ in $E_e$ is 
$$
a_{E_e}(X,B)=\langle \psi,e\rangle.
$$

Since $\langle \psi,e_i\rangle=a_i\ge 0$ for all $i$, we have $\psi\in \sigma^\vee$ (i.e. $\mld(X,B)\ge 0$). 

The moment polytope of $-K-B$ is 
$
\square_{-K-B}=-\psi+\sigma^\vee\subset M_\R.
$ 
Its polar set $U=\square^*\subset N_\R$ takes the explicit form
$$
U=\{e\in \sigma;\langle \psi,e\rangle\le 1\}=\Conv(\{0\}\cup\{\frac{e_i}{a_i};a_i>0\})+\sum_{a_i=0}\R_{\ge 0}e_i.
$$ 
The cone $\sigma_0=\sum_{a_i=0}\R_{\ge 0}e_i$ coincides with the face $\sigma\cap \psi^\perp$ of $\sigma$.
The notations are compatible with those in section 3, since we may view the germ of singularity as a germ of fibration
$(X/X\ni P,B)$, where $X\to X$ is the identity. For $t>0$, property $(C_t)$ is equivalent to $\mld_P(X,B)\ge t$, or 
$\min\{\langle \psi,e\rangle;e\in N\cap \Int\sigma\}\ge t$, or 
$N\cap \Int(tU)=\emptyset$.
Theorem~\ref{bndser} takes the following form:

\begin{thm}\label{bndgerm}
	Fix $d\ge 1$ and an ACC set $\cA\subset (0,+\infty)$. 
	Let $P\in (X,B)$ be a germ of toric log variety, let $t>0$. Suppose $\dim X=d$ and $\frac{a_i}{t}\in \cA$ for all $i$. The following are equivalent:
	\begin{itemize}
		\item[a)] $\mld_P(X,B)\ge t$.
		\item[b)] There exists a non-zero surjective homomorphism of lattices $\Phi\colon N\to N'$ such that the image $U'=\Phi(U)\subset N'_\R$
		is a compact polytope, $N'\cap \Int(tU')=\emptyset$ and $(N',tU')$ is bounded.
	\end{itemize}
\end{thm}

Note that while $0$ is an extremal point of $U$, the image $U'$ contains $0$ on its boundary, but usually not as an extremal point. We call $\Phi\colon (N,tU)\to (N',tU')$ a {\em $t$-lc reduction} of $P\in (X,B)$. If we denote $a(N',U')=\max\{t>0;N' \cap \Int(tU')=\emptyset \}$, then 
$$
\mld_P(X,B)=a(N,U) \ge a(N',U')\ge t.
$$

If $a_i=0$, then $\Phi(e_i)=0$. Suppose $\Phi(e_i)\ne 0$. Then $\Phi(e_i)=q_ie'_i$ with $q_i\in \Z_{\ge 1}$ and $e'_i\in {N'}^{prim}\cap \sigma'$. Here $\sigma'=\cup_{t\ge 0}tU'\subset N'_\R$ is the image of 
$\sigma=\cup_{t\ge 0}tU$. We have 
$$
\frac{\Phi(e_i)}{a_i}=\frac{e'_i}{a'_i}, a'_i=\frac{a_i}{q_i}.
$$
The integers $q_i$ belong to a finite set, $\frac{a'_i}{t}$ belong to the ACC set $\cA\cdot \{\frac{1}{q};q\in \Z_{\ge 1}\}$ and are bounded away from zero, and $\{e'_i;\Phi(e_i)\ne 0\}$ belong to a finite set up to $\Aut(N')$. In particular, $(N',\sigma')$ belongs to finitely many isomorphism types. If $\cA$ may only accumulate to $0$, then $\frac{a'_i}{t}$ belong to a finite set and $(N',tU')$ belongs to finitely many isomorphism types.

\begin{thm}\label{ses}
	Fix $d\ge 1$ and an ACC set $\cA\subset (0,+\infty)$. 
	Let $P\in (X,B)$ be a germ of toric log variety with $a:=\mld_P(X,B)>0$. Suppose $\dim X=d$ and $\frac{a_i}{a}\in \cA$ for all $i$. Then:
	\begin{itemize}
		\item[1)] $\frac{a_i}{a}$ belong to a finite set.
		\item[2)] Let $\Phi\colon (N,aU)\to (N',aU')$ be an $a$-lc reduction. Then 
		$(N',aU')$ belongs to finitely many isomorphism types, and corresponds to an affine toric log variety with invariant point 
		with $\mld_{P'}(X',\sum_j (1-a'_j)E'_j)=a$ and $\frac{a'_j}{a}\in \cA\cdot\{\frac{1}{q};q\in \Z_{\ge 1}\}$ for all $j$.
	\end{itemize}
	
\end{thm}

\begin{proof} 2) Let $\Phi\colon (N,aU)\to (N',aU')$ be an $a$-lc reduction. We claim that there exists a linear form $\psi'\colon N'_\R \to \R$
	such that $\psi=\psi'\circ \Phi$. Indeed, by assumption $N\cap \Int(aU)=\emptyset$ and there exists $e\in N^{prim}\cap \Int\sigma$ with
	$\langle \psi,e\rangle=a$. Thus $e$ belongs to the unique proper face of $aU$ which does not contain the origin. Let $V=\Ker(N_\R\to N'_\R)$.
	We have $(e+V)\cap \Int(aU)=\emptyset$, since the $\Phi$-image is contained in $\{\Phi(e)\}\cap \Int(aU')\subset N'\cap \Int(aU')=\emptyset$.
	Let $v\in V$. Since $e\in \Int\sigma$, there exists $\epsilon>0$ such that
	$e\pm \epsilon v \in \Int\sigma$. The condition $e\pm \epsilon v\notin \Int(aU)$ becomes $\langle \psi,e\pm \epsilon v\rangle\ge a$, i.e. 
	$\langle \psi,v\rangle=0$. Therefore $\psi\colon N_\R\to \R$ vanishes on $\Ker\Phi$.
	
	We obtain $U'=\{e'\in \sigma';\langle \psi',e'\rangle\le 1\}$. 
	Since $U'$ is compact, $\{0\}=\{e'\in \sigma';\langle \psi',e'\rangle=0\}$.
	Therefore $\psi'\in \Int({\sigma'}^\vee)$ and $0$ is a face of $\sigma'$.
	Note that $\sigma'-\sigma'=N'_\R$, since $\sigma-\sigma=N_\R$.
	
	Denote $e'=\Phi(e)\in N\cap \Int\sigma'$. We have $\langle \psi',e'\rangle=\langle \psi,e\rangle=a$. Therefore $e'$ belongs to the unique
	proper face of $aU'$ which does not contain $0$. In particular, $e'\in {N'}^{prim}\cap aU'$.
	
	Let $\sigma'(1)=\{e'_j;j\}$. Fix $j$ and denote $a'_j=\langle \psi',e'_j\rangle$. 
	There exists $i$ such that $\Phi(\R_+e_i)=\R_+e'_j$. Then 
	$\Phi(e_i)=q_ie'_j$ for some $q_i\in \Z_{\ge 1}$, which is finite from the above argument.
	We obtain $a'_j=\frac{a_i}{q_i}$.
	
	Note that since $\psi$ factors through $\psi'$, for each $e_i\in \sigma(1)$ we have $a_i>0$ if and only if $\Phi(e_i)\ne 0$.

	Set $X'=T_{N'}\emb(\sigma')$, $B'=\sum_j(1-a'_j)E'_j$ and $P'$ the unique closed point of $X'$.
	
	1) Fix $e_{i_0}\in \sigma(1)$ with $a_{i_0}>0$. There exists a subset $i_0\in I\subset \sigma(1)$ such that 
	$\{\Phi(e_i);i\in I\}$ are linearly independent and $e'=\sum_{i\in I}x_i\Phi(e_i)$ for some $x_i>0$. 
	Indeed, $e'\in \Int\sigma'$ and $\Phi(e_{i_0})\in \sigma'\setminus 0$ gives $v':=e'-x_{i_0}\Phi(e_{i_0})\in \partial\sigma'$
	for some $x_{i_0}>0$. Let $\tau'$ be the face of $\sigma'$ which contains $v'$ in its relative interior.
	Then $v'=\sum_{j\in J}x_je'_j$ where $\{e'_j;j\in J\}\subset \tau'(1)$ is a subset of linearly independent elements
	and $x_j>0$ for all $j\in J$. Since $\Phi(e_{i_0})\notin \tau'$, it follows that $\Phi(e_{i_0}), e'_j (j\in J)$ are 
	linearly independent. For each $j\in J$, there exists $e_{i(j)}\in \sigma(1)$ such that $\Phi(e_{i(j)})=q_je'_j$
	for some $q_j\in \Z_{\ge 1}$. Set $I=\{i_0\} \cup \{i(j);j\in J\}$ and $x_{i(j)}=\frac{x_j}{q_j}$.
	This proves the claim.
	
	We have $e'=\sum_{i\in I}x_i\Phi(e_i)$ with $x_i>0$ for all $i\in I$. Since 
	$\langle \psi',e'\rangle=\min\{\langle \psi',y\rangle; y\in N'\cap \Int\sigma'\}$,
	we have $e'-\Phi(e_i)\notin \Int\sigma'$, hence $x_i\le 1$. Therefore $x_i\in (0,1]$ for all $i\in I$.
	In some basis of $N'$, the vectors $e'$ and $\Phi(e_i)$ have finitely many coordinates. By Cramer's rule, we obtain
	$rx_i\in \Z$ for all $i\in I$, where $r$ is a bounded positive integer. Therefore $\{x_i;i\in I\}$ is contained
	in a finite set.
	
	We have $\Phi(e)=\sum_{i\in I}x_i\Phi(e_i)$. Therefore $a=\sum_{i\in I}x_ia_i$, i.e.
	$$
	1=\sum_{i\in I}x_i\frac{a_i}{a}.
	$$
	Since $x_i$ are finite and $\frac{a_i}{a}$ belong to an ACC set, it follows that $\frac{a_i}{a}$ belong to a finite set.
	Therefore $\frac{a_{i_0}}{a}$ belongs to a finite set.
\end{proof}

\begin{cor} 
	Let $P\in (X,B)$ be an affine toric log variety with invariant point. Suppose $\dim X=d$ is fixed and the coefficients of $B$ belong to a fixed DCC set. Then:
	\begin{itemize}
		\item[1)] $\mld_P(X,B)$ belongs to an ACC set.
		\item[2)] Suppose $\mld_P(X,B)=t$ is fixed. Then the coefficients of $B$ belong to a finite set. If moreover $B$ has rational coefficients, then $n(K+B)\sim 0$ for some bounded positive integer $n$.
	\end{itemize}
\end{cor}

\begin{proof} Suppose $\mld_P(X,B)=a>0$ belongs to a DCC set.
Let $B=\sum_i(1-a_i)E_i$. Since $a_i$ satisfy ACC, $\frac{a_i}{a}$ belongs to an ACC set. By Theorem~\ref{ses}, $\frac{a_i}{a}$ belong to a finite set. Since only finite sets satisfy both ACC and DCC, $a_i,a$ belong to a finite set.

Suppose moreover $B$ has rational coefficients. Let $n\ge 1$ be minimal
such that $n(K+B)\sim 0$.  By Theorem~\ref{ses}, $n$ is also the index of an $a$-lc reduction $P'\in (X',B')$, which belongs to finitely many isomorphism types. Therefore $n$ is bounded.
\end{proof}

Theorem~\ref{bndmult} takes the following form for germs of toric singularities: suppose $\mld_P(X,B=\sum_i(1-a_i)E_i)\ge t>0$, $\dim X=d$ is fixed and $\frac{a_i}{t}$ belongs to an ACC set. Then there exists an invariant hyperplane section $H\in |\fm_{X,P}|$ such that $\mld(X,B+\gamma H)\ge 0$ for some
$\gamma>0$ with $\frac{\gamma}{t}$ bounded away from zero.

Theorems~\ref{bndscoP} and~\ref{bndscogl} take the following form for germs of toric singularities: suppose $P\in (X,B)$ has $r$-hyperstandard boundary. If $\mld_P(X,B)\ge t$ (resp. $\mld(X,B)\ge t$), there exists
$B^+\ge B$ such that $n(K+B^+)\sim 0$ and $\mld_P(X,B^+)\ge t$ (resp. $\mld(X,B^+)\ge t$), where $n$ is a positive integer which depends only on $\dim X, r,t$.

%%%%%%%%%%%%%%%%%%%%%%%%%%%%%%%%%%%%%%%

\subsection{Case $X$ is $\Q$-factorial}

%%%%%%%%%%%%%%%%%%%%%%%%%%%%%%%%%%%%%%%
	
The affine toric variety $X=T_N\emb(\sigma)$ is $\Q$-factorial if and only if $\sigma(1)$ has exactly $d=\dim X$ elements.
Up to isomorphism, we may suppose $e_1,\ldots,e_d$ are the standard basis of $\Z^d$, and $N\subset \R^d$ is a $d$-dimensional
lattice which contains $e_i$ as primitive elements. Note that $N/\Z^d$ is a finite abelian group, $\sigma=\R_{\ge 0}^d$ and 
$$
U=\{x\in \R^d_{\ge 0};\sum_{i=1}^da_ix_i\le 1\},
\
tU=\{x\in \R^d_{\ge 0};\sum_{i=1}^d\frac{a_i}{t} x_i\le 1\}.
$$
We claim that Theorem~\ref{bndgerm} is equivalent to the following statement: let $t>0$, suppose $\frac{a_i}{t}$ belong to a 
fixed ACC set $\cA$. Then there exist finitely many closed subgroups 
$\Z^d \le G\le \R^d$ with $G\cap \Int(tU)=\emptyset$, such that $\mld_P(X,B)\ge t$ if and only if $N$ is contained in some $G$.

Indeed, let $(N,tU)\to (N',tU')$ be a $t$-lc reduction. We may find an isomorphism $N'\isoto \Z^p$ with $1\le p\le d$
such that $tU'$ is mapped into $[-\lambda,\lambda]^p$, where $\lambda=\lambda(d,\cA)$ is bounded above. 
Let $g\colon \R^d=N_\R\to N'_\R\isoto \R^p$ be the induced homomorphism, let $G=g^{-1}(\Z^p)$. Since $g$ maps
$\Int(tU)$ into $\Int(tU')$, we obtain $G\cap \Int(tU)=\emptyset$. Since $g(N)=\Z^p$, we obtain $\Z^d\le N\le G$.
Since $g$ maps $tU$ into $[-\lambda,\lambda]^p$ and $e_i\in a_iU$, the components of the matrix defining $g$ are integers $z$ with $| z | \le  \frac{a_i}{t}\lambda$. Since $\frac{a_i}{t}$ satisfy ACC, it is bounded above.
Therefore $| z|$ is bounded above. Therefore $g$ belongs to a finite set of maps, hence $G$ belongs to finitely many groups, depending only on $d$ and $\cA$.

Conversely, suppose $N$ is contained in some $G$ belonging to a finite family. Since $\Z^d\le G\le \R^d$, the dual group $G^*=\Hom_\Z(G,\Z)$ is 
a subgroup of $(\Z^d)^*$. It is non-zero since $tU$ is non-empty, hence 
$G\ne \R^d$. Since $(\Z^d)^*$ is a free abelian group of rank $d$, it follows that $G^*$ is a free abelian group of rank $p$, for some $1\le p\le d$. Since $G$ belongs to a finite family, we may find a basis 
$\varphi_1,\ldots,\varphi_p$ of $G^*$ such that their coordinates in the standard dual basis of $(\Z^d)^*$ have bounded absolute value. Let $g\colon \R^d\to \R^p$ be the surjective homomorphism with components $(\varphi_1,\ldots,\varphi_p)$. Since $G=G^{**}$, we have 
$G=g^{-1}(\Z^p)$. Let $U'=g(U)$. The assumption $G\cap \Int(tU)=\emptyset$
is equivalent to $\Z^p\cap \Int(tU')=\emptyset$. Since the entries of the matrix defining $g$ are bounded, $U$ is the convex hull of $0$ and 
$\frac{a_i}{t}e_i$, and $\frac{a_i}{t}$ are bounded above, it follows that $tU'$ is contained in $[-\lambda,\lambda]^p$ for some $\lambda=\lambda(d,\cA)$. Also, $\Z^p/g(\Z^d)$ is a finite abelian group, of cardinality bounded above. Therefore the same holds for its quotient
$\Z^p/g(N)$.  If we set $N'=g(N)\subseteq \Z^p$, then $g\colon (N,tU)\to (N',tU')$ is a $t$-lc reduction.

%\clearpage
%%%%%%%%%%%%%%%%%%%%%%%%%%%%%%%%%%%%%%%
%%%%%%%%%%%%%%%%%%%%%%%%%%%%%%%%%%%%%%%

\end{document}